\documentclass[12pt]{amsart}
\usepackage{etex}
\usepackage[active]{srcltx}
\usepackage{calc,amssymb,amsthm,amsmath,amscd, eucal,ulem}
\usepackage{alltt}
\usepackage{mathtools}
\usepackage{fancyhdr}
\usepackage{enumerate}

\usepackage{bbding}
\RequirePackage[dvipsnames,usenames]{color}

\input{mabliautoref.sty}
\input{xy}
\xyoption{all}
\input{kmacros3.sty}
\usepackage{tikz}
\usepackage{amsfonts, mathrsfs}
\usepackage[cal=boondox, calscaled=1.05]{mathalfa}
\usepackage{calligra}
\usepackage{amssymb}
\usepackage{stmaryrd} 
\usepackage{dcpic, pictexwd}
\usepackage[left=1in,top=1in,right=1in,bottom=1in]{geometry}
\usepackage{bm}
\usepackage{verbatim}
\usepackage{upgreek}
\usepackage{systeme}

\numberwithin{equation}{theorem}

\newcommand{\picl}[0]{\operatorname{Pic}^{\rm loc}}
\newcommand{\picls}[0]{\operatorname{\mathbf{Pic}}^{\rm loc}}

\newcommand{\Ex}[0]{\operatorname{Ex}}

\DeclareMathOperator{\ssPic}{ \sP\!\!\;\!\text{\calligra{\Large ic}\,}}
\newcommand{\kay}{\mathcal{k}}

\renewcommand{\:}{\colon}
\newcommand{\eg}{{\itshape e.g.} }

\newcommand{\p}{\mathfrak{p}}

\DeclareMathOperator{\depth}{depth}

\newcommand{\Gal}{\textnormal{Gal}}

\newcommand{\rdown}[1]{\lfloor{#1}\rfloor}
\newcommand{\diff}[0]{\operatorname{Diff}} 

\DeclareMathOperator{\Cl}{Cl}
\DeclareMathOperator{\NS}{NS}

\usepackage{setspace}
\usepackage{hyperref}
\usepackage{enumerate}
\usepackage{enumitem}
\usepackage{graphicx}
\usepackage[all,cmtip]{xy}

\usepackage{verbatim}

\theoremstyle{theorem}

\renewcommand{\sF}{\mathcal{F}}

\renewcommand{\sH}{\mathcal{H}}

\renewcommand{\sL}{\mathcal{L}}

\renewcommand{\sO}{\mathcal{O}}
\renewcommand{\sP}{\mathcal{P}}

\usepackage{marvosym}

\makeatletter
\let\@wraptoccontribs\wraptoccontribs
\makeatother

\renewcommand\headrulewidth{0pt}

\begin{document}
\title[\'Etale fundamental groups of KLT threefold singularities]{On the local \'etale fundamental group of KLT threefold singularities} 
\author[J.~Carvajal-Rojas]{Javier Carvajal-Rojas}
\address{KU Leuven\\ Department of Mathematics\\ Celestijnenlaan 200B \\3001 Heverlee\\Belgium \newline\indent
Universidad de Costa Rica\\ Escuela de Matem\'atica\\ San Jos\'e 11501\\ Costa Rica}
\email{\href{mailto:javier.carvajalrojas@epfl.ch}{javier.carvajal-rojas@kuleuven.be}}
\author[A.~St\"abler]{Axel St\"abler}
\address{Universit\"at Leipzig\\ Institut f\"ur Mathematik\\Augustusplatz 10\\
04109 Leipzig\\Germany} 
\email{\href{mailto:staebler@uni-leipzig.de}{staebler@math.uni-leipzig.de}}

\contrib[with an appendix by]{J\'anos Koll\'ar}
\address{Princeton University\\Princeton NJ 08544-1000\\USA} 
\email{\href{mailto:kollar@math.princeton.edu}{kollar@math.princeton.edu}}


\thanks{
The first named author was partially supported by the grants ERC-STG \#804334 and FWO \#G079218N. The second named author was partially supported by SFB-Transregio 45 Bonn-Essen-Mainz financed by Deutsche Forschungsgemeinschaft.}


\begin{abstract}
Let $S$ be KLT threefold singularity over an algebraically closed field of positive characteristic $p>5$. We prove that its local \'etale fundamental group is tame and finite. Further, we show that every finite unipotent torsor over a big open of $S$ is realized as the restriction of a finite unipotent torsor over $S$.
\end{abstract}
\maketitle

\section{Introduction}
Consider the following setup:

\begin{setup} \label{setup}
Let $(X,\Delta)$ be a Kawamata log terminal (KLT fo short) log variety of dimension $d \geq 2$ and defined over a field $\kay$ of characteristic $p \geq 0$. Let $\bar{x} \to X$ be a geometric closed point.
Set $R \coloneqq \sO_{X,\bar{x}}^{\mathrm{sh}}$ and denote its field of fractions by $K$. Let $Z \subset S \coloneqq \Spec R$ be a closed subscheme of codimension at least $2$ and set $S^{\circ} \coloneqq S \smallsetminus Z$. In other words, $S^{\circ}$ is an arbitrary big open subset of $S$. We choose and fix a geometric generic point $\bar{\eta} \to S^{\circ}$ (i.e., a separable closure of $K$) and let $\pi_1^{\mathrm{\acute{e}t}}(S^{\circ}) = \pi_1^{\mathrm{\acute{e}t}}(S^{\circ}, \bar{\eta})$ denote the \'etale fundamental group of $S^{\circ}$ with base point $\bar{\eta}$. If $Z$ cuts out the singular locus, we may refer to it as the (\'etale) fundamental group of the singularity.
\end{setup}
In this paper, we are concerned with the problem of determining finiteness of $\pi_1^{\mathrm{\acute{e}t}}(S^{\circ})$. In \cite{XuFinitenessOfFundGroups}, C.~Xu proved that if $\kay = \mathbb{C}$ then $\pi_1^{\mathrm{\acute{e}t}}(S\smallsetminus\{x\})$ is finite; \cf \cite{GrebKebekusPeternellEtaleFundamental, TianXuFinitenessOfFundamentalGroups, BraunFundamentalGroupKLTTopological}. Inspired by this result, K.~Schwede, K.~Tucker, and the first named author proved that if $p>0$ and $S$ is strongly $F$-regular then $\pi_1^{\mathrm{\acute{e}t}}(S^{\circ})$ is finite and its order is at most $1/s(R)$ and prime to $p$; see \cite{CarvajalSchwedeTuckerEtaleFundFsignature}, \cf \cite{CarvajalFiniteTorsors}.\footnote{Here, $s(R)$ denotes the $F$-signature of $R$.
} It is worth recalling that Bhatt--Gabber--Olsson later introduced in \cite{BhattGabberOlssonFinitenessFun} a spreading out technique to prove Xu's results from the one in \cite{CarvajalSchwedeTuckerEtaleFundFsignature}. We also recommend to see \cite{JeffriesSmirnovTransformationRuleForNaturalMultiplicities}, where a characteristic free approach is discussed in terms of differential signatures. However, to the best of the authors' knowledge, little is known about the finiteness of $\pi_1^{\mathrm{\acute{e}t}}(S^{\circ})$ if $S$ is a non-strongly-$F$-regular KLT singularity in positive characteristic---not even in sufficiently large characteristic. In this work, we attempt to address this case in dimension $3$ and $p>5$. In positive characteristic, there is another aspect regarding $\pi_1^{\mathrm{\acute{e}t}}(S^{\circ})$ that is worth investigating, namely, its \emph{tameness}. For strongly $F$-regular $S$, we know that this group is very tame as $p$ does not divide its order \cite[Corollary 2.11]{CarvajalSchwedeTuckerEtaleFundFsignature}, \cf \cite[Theorem F]{CarvajalFiniteTorsors}. The same cannot be true for general KLT $S$ by the (independent) examples of T.~Yasuda and B.~Totaro; see \cite{YasudaDiscrepanciesPcyclicQuotientVarieties,TotaroFailureKVforFanos} respectively.

It is worth observing that in dimension $2$ and characteristics $p>5$, KLT singularities and strongly $F$-regular singularities coincide; see \cite{HaconXuMMPPostiveChar,HaraDimensionTwo}. Therefore, if the characteristic is large enough $(p>5)$, the aforementioned problem is solved by \cite{CarvajalSchwedeTuckerEtaleFundFsignature}, \cf \cite{ArtinCoveringsOfTheRtionalDoublePointsInCharacteristicp,ArtinWildlyRamifiedZ2Actions}.\footnote{To the best of the authors's knowledge, this is still open if $p\in \{2,3,5\}$ and only known for canonical singularities using Artin--Lipman's explicit classification of rational double points \cite{ArtinCoveringsOfTheRtionalDoublePointsInCharacteristicp,LipmanRationalSingularities}. The problem is that we do not know whether canonical covers of log terminal surface singularities remain log terminal (when the characteristic does divide the Gorenstein index).} This is no longer the case in dimension $3$ as there exist non-strongly-$F$-regular but KLT threefold singularities; see \cite{CasciniTanakaWitaszekKLTdelPezzoNotFregular}.

In dimension $3$ and characteristic $p>5$, we recall that C.~Xu and L.~Zhang have demonstrated in \cite[Theorem 3.4]{XuZhangNonvanishing} that a certain tame quotient of $\pi_1^{\mathrm{\acute{e}t}}(S\smallsetminus\{x\})$ is finite. Their methods are for the most part identical to the ones in \cite{XuFinitenessOfFundGroups}, which can only handle the tame part. In this article, we complement Xu--Zhang's work on the matter by proving that $\pi_1^{\mathrm{\acute{e}t}}(S^{\circ})=\pi_1^{\mathrm{\acute{e}t}}(S\smallsetminus Z)$ is finite and tame for arbitrary $Z$. 
If $G$ is a profinite group, we denote by $G^{(p)}$ the prime-to-$p$ part of $G$, which is the inverse limit over all finite quotients of $G$ whose order is not divisible by $p$. Then, our main result is the following.

\begin{theoremA*}[{\autoref{thm.TamenessFundGroup}, \autoref{theo.mainresultbody}}]
Working in \autoref{setup}, assume $d = 3$ and $\kay$ to be algebraically closed of characteristic $p>5$. Then, the canonical quotient $\pi_1^{\mathrm{\acute{e}t}}(S^{\circ}) \twoheadrightarrow \pi_1^{\mathrm{\acute{e}t}}(S^{\circ})^{(p)}$ is an isomorphism between finite groups.
\end{theoremA*}

\begin{remark}
By the purity of the branch locus theorem, we may always assume that $Z$ cuts out the singular locus of $S$. Further, since $S^{\circ }$ is a $2$-dimensional KLT scheme (of characteristic $p>5$), it has finitely many singular points which all have finite fundamental group. Hence, we may find a Galois quasi-\'etale cover $S'=\Spec R' \to S$ such that every quasi-\'etale cover over the punctured spectrum of $R'$ is \'etale. In other words, we may also assume that $Z=\{x\}$. However, we do not do this here as we want our methods to generalize to higher dimensions.
\end{remark}

Just as in \cite{BhattCarvajalRojasGrafSchwedeTucker}, we obtain the following application.

\begin{theoremB*}[{\autoref{cor.DirectCorollary}, \autoref{cor.FinitenessAndTamenessOfFundGroupsOfBigOpens}}]
Working in \autoref{setup}, assume $d = 3$ and $\kay$ to be algebraically closed of characteristic $p>5$. Then, every quasi-\'etale cover over $X$ whose degree is a power of $p$ is \'etale. Moreover, there exists a quasi-\'etale Galois cover $\tilde{X} \to X$ with prime-to-$p$ degree such that every quasi-\'etale cover over $\tilde{X}$ is \'etale. Moreover, if $(X,\Delta)$ is log Fano and $U \subset X$ is a big open then $\pi_1^{\mathrm{\acute{e}t}}(U)$ is finite and its order is prime-to-$p$.
\end{theoremB*}

Given a scheme $X$, we denote by $X_{\mathrm{fl}}$ the (small) fppf site of $X$, \ie the flat topology on $X$ \cite{MilneEtaleCohomology}. In particular, $H^1(X_{\mathrm{fl}},G)$ represents the isomorphism classes of (fppf) $G$-torsors over $X$, where $G$ is a group-scheme over $X$. See\cite[III, \S4]{MilneEtaleCohomology} for further details. If $\pi_1^{\mathrm{\acute{e}t}}(S^{\circ}) \twoheadrightarrow \pi_1^{\mathrm{\acute{e}t}}(S^{\circ})^{(p)}$ is an isomorphism, the restriction map of torsors $H^1(S_{\mathrm{fl}}, G) \to H^1(S^{\circ}_{\mathrm{fl}},G)$ is surjective for all finite \'etale unipotent group-schemes $G/\kay$ (i.e. for the $p$-groups). Inspired by \cite[Theorem F]{CarvajalFiniteTorsors}, we push this further by using the rationality hypothesis.
\begin{theoremC*}[\autoref{cor.fulltameness}]
Working in \autoref{setup}, assume $d = 3$, and $\kay$ to be algebraically closed of characteristic $p>5$. If $S$ is rational, then the restriction map of torsors
\[
\varrho^1_S(G) \: H^1(S_{\mathrm{fl}}, G) \to H^1(S^{\circ}_{\mathrm{fl}},G)
\]
is surjective for all finite unipotent group-schemes $G/\kay$.
\end{theoremC*}

The problem of the rationality of KLT singularities in positive characteristics has been of great interest in recent years, where several examples of non-rational KLT singularities have been found.\footnote{In fact, not Cohen--Macaulay. However, for KLT singularities, Cohen--Macaulayness and rationality are equivalent notions; see \cite{KovacsRationalSingularities}.} There are essentially two known ways to construct such counterexamples. On one hand, we have the examples obtained from (by taking affine cones over) counterexamples of Kodaira or Kawamata--Viehweg vanishing theorems; \eg \cite{GongyoNakamuraTanakaRationalPointsFanoThreefoldsFiniteFields,CasciniTanakaPLTThreefoldsNonNormalCenters,KovacsNonCMCanonicalSingularities,BernasconiKVFailsLogDelPezzo,TotaroFailureKVforFanos}. On the other hand, we have examples constructed by taking quasi-\'etale $\bZ/p\Z$-quotients of regular germs; see \cite{YasudaDiscrepanciesPcyclicQuotientVarieties,YasudaPcyclicMackayCorrespondenceMotivicIntegration,TotaroFailureKVforFanos}. In all these examples, either the dimension is large or the characteristic is small, which is predicted by \cite{HaconWitaszekRationalityKawamataLogTerminalPosChar} where it is shown that KLT threefold singularities are rational in sufficiently large characteristics. Further, in Yasuda's construction of certain quasi-\'etale $\bZ/p\Z$-quotients $\hat{\bA}_{\kay}^d \to S$, we see that $S$ cannot be simultaneously rational (i.e. Cohen--Macaulay) and KLT.\footnote{In fact, in dimension $3$, these are never KLT.} In those examples, $\pi_1^{\mathrm{\acute{e}t}}(S) = \bZ/p\bZ$ as $\hat{\bA}_{\kay}^d \to S$ is a univeral Galois cover. Our theorem goes further by establishing that KLT threefold singularities in characteristic $p>5$ cannot be realized as wild \'etale quotients of regular germs. We also make the following observation.

\begin{theoremD*}[\autoref{cor.FrationalityImpliesRational}]
Working in \autoref{setup}, assume $d = 3$ and $\kay$ to be algebraically closed of characteristic $p>5$. If $S$ is $F$-injective then it is rational.
\end{theoremD*}

\begin{remark} \label{rem.Rationality}
When a first version of this work was made public, \cite{ArvidssonBernasconiLaciniKVVforLogDelPezzosp>5} was not available. In that paper, rationality of KLT threefold singularities is established in characteristic $p>5$. This makes, for example, Theorem D obsolete and together with Theorem C answers \autoref{3rd.question}; see \cite{ArvidssonBernasconiLaciniKVVforLogDelPezzosp>5}. However, we have decided to preserve our original paper as much as possible in this second version. In particular, we keep emphasizing what type of results need rationality and which need not. On the other hand, by using the \autoref{sec.Appendix} kindly provided by J\'anos~Koll\'ar, we are able to remove the use of rationality from our original arguments regarding finiteness of the fundamental group.
\end{remark}

\subsection{Outline of the paper}
Throughout this work, we follow the ideas in \cite{HaconWitaszekRationalityKawamataLogTerminalPosChar} to study KLT threefold singularities in characteristic $p>5$ by strongly $F$-regular ones via the so-called generalized PLT blowups. Note that all theorems we have claimed in this introduction are true for strongly $F$-regular singularities. What we do for the most part of this paper is to prove that the same properties can be transferred across generalized PLT blowups which; loosely speaking, is what Hacon--Witaszek do in \cite{HaconWitaszekRationalityKawamataLogTerminalPosChar} for rationality. 

In \autoref{sec.preliminaries}, we investigate the problem of (unipotent) tameness. This problem has to do with the action of Frobenius on (local) cohomology. Therefore, what we do in \autoref{sec.preliminaries} is to compare the action of Frobenius on cohomology with supports along a generalized PLT blowup. 

With tameness in place, we prove finiteness of the fundamental group in \autoref{sec.finiteness}. Our proof starts identical to Xu's argument in \cite{XuFinitenessOfFundGroups} (which was later mimicked in \cite[\S 3]{XuZhangNonvanishing}). That is, we start off by considering a PLT blowup $(T,E) \to (S,x)$ at our KLT point $x$ to extract a Koll\'ar component $E$. Then, the tame fundamental group of $T$ with respect to $E$; say $\pi_1^{\mathrm{t},E}(T)$,  can be used to study tame quasi-\'etale covers over $S$. More precisely, we use Grothendieck--Murre's \cite[Corollary 9.8]{GrothendieckMurreTameFundamentalGroup}; see \autoref{cla.GrotehndieckMurre}, to establish an isomorphism between $\pi_1^{\mathrm{t},E}(T)^{(p)}$ and $\pi_1^{\mathrm{\acute{e}t}}(S)^{(p)}$. From here, both approaches diverge significantly. On one hand, Xu--Zhang use the Minimal Model Program in positive characteristics to study finiteness of $\pi_1^{\mathrm{t},E}(T)^{(p)}$. Our approach, however, follows the philosophy in \cite{HaconWitaszekRationalityKawamataLogTerminalPosChar} in the sense that we rather exploit that $(T,E)$ has purely $F$-regular singularities.
More precisely, we use that $(T,E)$ has purely $F$-regular singularities to prove that $\pi_1^{\mathrm{t},E}(T)^{(p)}$ is finite via a local-to-global argument and Galois-theoretic considerations. The local part is the content of the authors' former preprint \cite{CarvajalRojasStaeblerTameFundamentalGroupsPurePairs} whereas we explain the local-to-global methods in \autoref{sec.LocalToGlobal}. In the end, just as in \cite{XuFinitenessOfFundGroups} and \cite{XuZhangNonvanishing}, we are left with studying the structure of totally ramified tame covers over $(T,E)$. We do this by examining the local structure of these covers via our singular Abhyankar's lemma in \cite{CarvajalRojasStaeblerTameFundamentalGroupsPurePairs}. This reduces the problem to study the finiteness of the prime-to-$p$ torsion of the divisor class group of $S$. This may be easily obtained by using \cite{BoutotSchemaPicardLocal} under the rationality hypothesis. For the reader's convenience, we provide details in \autoref{sec.LocalPicardSchemes}. However, this has been substantially improved by J\'anos Koll\'ar in \autoref{sec.Appendix}, where the rationality hypothesis is not employed directly.

\begin{convention}
In this paper, all schemes and rings are defined over $\mathbb{F}_p$. We denote the $e$-th iterate of the Frobenius endomorphism by $F^e \: X \to X$. We assume all our schemes and rings to be $F$-finite, noetherian, and so excellent.
\end{convention}

\subsection*{Acknowledgements} 
 We would like to thank Emelie Arvidsson, Manuel Blickle, Zsolt Patakfalvi, Thomas Polstra, Karl Schwede, Jakub Witaszek, and Maciej Zdanowicz for very useful discussions and help throughout the preparation of this preprint. We are very grateful to Thomas Polstra for pointing out some mistakes on an earlier draft of this work and for helping us correcting them. We thank Maciej Zdanowicz for his help in filling some gaps in some of our early arguments. We are particularly thankful to Jakub Witaszek whose ideas inspired the authors to work on this project and for his sincere interest on previous drafts of this preprint which he kindly helped to improve and correct. In fact, it was a conversation with him after the first named author's thesis defense which motivated the whole project. We also thank the anonymous referee for many valuable comments and corrections.
 
\section{Generalized PLT blowups and Tameness} \label{sec.preliminaries}
We commence by recalling the definition of Koll\'ar components and PLT blowups; see \cite[\S 1.1]{LiXuStabilityofValuationsAndKollarComponents}, \cite[Lemma 1]{XuFinitenessOfFundGroups}. In dimension $3$ and positive characteristic $p>5$, we also recall the refined notion of generalized PLT blowups, which, roughly speaking, can be used to approximate KLT threefold singularities by strongly $F$-regular ones. After this, we explain why this approximation is good enough to impose tameness conditions on those KLT singularities.

\subsection{Generalized PLT blowups} Consider the following definition.

\begin{definition}[PLT blowups and Koll\'ar components] \label{def.KollarComponent}
Let $X$ be a variety and let $x \in X$ be a closed point. One says that a proper birational morphism $f \: Y \to X$ from a log pair $(Y,\Delta_Y)$ is a \emph{PLT blowup at $x$} if the following conditions hold:
\begin{enumerate}
    \item $f$ induces an isomorphism over $X \smallsetminus \{x\}$,
    \item $E \coloneqq f^{-1}(x)$ is a normal prime divisor on $Y$ such that $(Y,E+\Delta_Y)$ is a PLT log pair,
    \item  $-(K_Y+E+\Delta_Y)$ is a $\mathbb{Q}$-Cartier $f$-ample divisor, and
    \item $-E$ is a $\mathbb{Q}$-Cartier $f$-nef divisor.
\end{enumerate}
The exceptional divisor $E$ of a PLT blowup $f \: Y \to X$ is called a \emph{Koll\'ar component} of $X$ at the closed point $x$.
\end{definition}

\begin{remark}
\label{rem.logfano}
Let $X$ be variety, $x \in X$ be a closed point, and $f \: Y \to X$ be a PLT blowup extracting a Koll\'ar component $E$. Further, if we set an equality
\[
(K_Y+ E +\Delta_Y)\big|_E = K_E + \Delta_E
\]
given by adjunction \cite[Chapter 4]{KollarSingulaitieofMMP}. Then, one observes that $(E,\Delta_E)$ is a KLT pair (\cite[Lemma 4.8]{KollarSingulaitieofMMP} using property (b)) and $-(K_E+\Delta_E)$ is ample by property (c). In other words, $(E,\Delta_E)$ is a log Fano pair. We may often refer to this pair as a Koll\'ar component of $X$ at $x$ as well. Furthermore, since $-(K_Y+E+\Delta_Y)$ is ample over $X$ and $(Y,E+\Delta_Y)$ is PLT, we have that $(X,f_*\Delta_Y)$ is KLT; see for instance \cite[Corollary 3.44]{KollarMori}.
\end{remark}

Regarding the existence of PLT blowups and Koll\'ar components, we may say the following. For further details, see \cite{PrhhorovBlowupsCanonicalSIngularities,XuFinitenessOfFundGroups,GongyoNakamuraTanakaRationalPointsFanoThreefoldsFiniteFields,HaconWitaszekMMPLowCharacteristic,HaconWitaszekMMPThreefoldsChar5}.
\begin{theorem} \label{thm.ExistencePLTBlowUP}
Let $x\in (X,\Delta)$ be a KLT closed point of a log pair of dimension $d\geq 2$ and characteristic $p$. Then, a Koll\'ar component at $x$ exists if either $p=0$ or $d \leq 3$.
\end{theorem}

The following result is crucial in this work. It establishes that for KLT threefolds in characteristic $p >5$ PLT blowups can even be arranged to be ``purely $F$-regular blowups.'' This result is extracted from \cite[Propositions 2.8 and 2.11]{HaconWitaszekRationalityKawamataLogTerminalPosChar}.

\begin{theorem}[Generalized PLT blowups] \label{thm.ExistencePFTRblowUP}
Let $(X, \Delta)$ be a KLT threefold over an algebraically closed field $\kay$ of characteristic $p > 5$ and $x \in X$ a closed point. Then, for a suitable open neighborhood $U$ of $x$, there exists a $\mathbb{Q}$-factorial PLT blowup $f\: (Y,\Delta_Y) \to U$ centered at $x$ extracting a Koll\'ar component $(E,\Delta_E) \subset Y$ such that $(Y,E)$ is a purely $F$-regular log pair.  Moreover, if $\Delta$ has standard coefficients then so does $(E,\Delta_E)$ and so $(Y,E+\Delta_Y)$ is purely $F$-regular.
\end{theorem}
\begin{proof}
Recall that the non-$\mathbb{Q}$-factorial points of $X$ are isolated. To wit, they form a closed subset by \cite[Theorem 6.1]{BoissiereGAbberSermanVarietesLocalmentFactorielles}. By \cite[Theorem 2.14]{GongyoNakamuraTanakaRationalPointsFanoThreefoldsFiniteFields}, there exists a \emph{small}\footnote{That is, it is an isomorphism in codimension-$1$} birational morphism $Y \to X$ where $Y$ is $\mathbb{Q}$-factorial. We conclude that the non-$\mathbb{Q}$-factorial locus must have codimension $\geq 3$.
Hence, we find an open neighborhood $U$ of $x$ such that each point except possibly $x$ is $\mathbb{Q}$-factorial. Now we can apply \cite[Theorem 2.11]{HaconWitaszekRationalityKawamataLogTerminalPosChar} (where we possibly shrink $U$ again).

For the final statement, simply use that KLT and strong $F$-regularity coincide for $2$-dimensional pairs with standard coeffcients \cite{HaraFRegFPureGraded, HaconXuMMPPostiveChar, CasciniGongyoSchwedeuniformbounds}. Hence, one may apply inversion of $F$-adjunction to conclude $(Y,E+\Delta_Y)$ is purely $F$-regular \cite{DasStronglyFregularInversionOfAdjunction,TaylorInvsersionAdjunctionFSignature}.
\end{proof}

As a consequence, Hacon--Witaszek proved the following.

\begin{theorem} \label{thm.RationalityVanishing}
With notation as in \autoref{thm.ExistencePFTRblowUP}, if Kawamata--Viehweg vanishing holds for $(E,\Delta_E)$, then $R^kf_* \sF = 0$ for all $k>0$ if $\sF$ is either $\sO_Y$ or $\omega_Y$ and so $\sO_{X,x}$ is Cohen--Macaulay and further rational. Conversely, if $\sO_{X,x}$ is rational then $R^kf_* \sF = 0$ for all $k>0$ if $\sF$ is either $\sO_Y$ or $\omega_Y$.
\end{theorem}
\begin{proof}
The first part is the content of \cite[\S4]{HaconWitaszekRationalityKawamataLogTerminalPosChar}. To see the converse, note that $Y$ is strongly $F$-regular, so $F$-rational, and further rational. Additionally, we may find a common resolution 
\[
\xymatrix{
Y \ar[d]_-{f} & Y' \ar[l]_-{g} \ar[ld]^-{h}\\
U
}
\]
where it is noted that $U$ has rational singularities if $x \in U$ is rational. Hence, 
\[
\mathbf{R}f_* \sO_Y=\mathbf{R}f_* \big(\mathbf{R}g_* \sO_{Y'}\big) = \mathbf{R} h_* \sO_{Y'} = \sO_U
\]
where we used rationality of $T$ and $U$ in the first and last equality; respectively. The same argument works replacing $\sO$ by $\omega$ using \cite[Lemma 8.2]{KovacsRationalSingularities}, \cf \cite[Proposition 2.3]{HaconWitaszekRationalityKawamataLogTerminalPosChar}.
\end{proof}

\subsection{Tameness}

We apply the above results to obtain tameness conditions for KLT threefold singularities in characteristic $p>5$. We need, however, some standard results on fundamental groups for log del Pezzo surfaces.

\subsubsection{\'Etale fundamental groups of big opens of log del Pezzo surfaces}
We remind the reader that a \emph{weak log del Pezzo surface} is a KLT surface $(X, \Delta)$ for which $-(K_X + \Delta)$ is nef and big.
We assume the following result is well-known to experts; see \cite{TakayamaSimpleConnectednessFanos,QiRationalConnectednessLogFanoes,GurjarZhangPi1SmoothPointsLogdelPezzo,XuNoetsPi1SmoothLociDelPezzo}. However, we include a proof for the sake of completeness. It will play a fundamental role in our study of tameness.

\begin{theorem} \label{thm.FundamentalGrouoDelPezzo}
Let $(X,\Delta)$ be a weak log del Pezzo surface over an algebraically closed field $\kay$ of positive characteristic $p$ ($X$ is in particular projective over $\kay$). Then, $\pi_1^{\mathrm{\acute{e}t}}(X) = 1$ and further, if $p>5$, $\pi_1^{\mathrm{\acute{e}t}}(U)$ is finite for all big opens $U \subset X$. 
\end{theorem}
\begin{proof}
The first statement is just as in \cite[Theorem 1.1]{TakayamaSimpleConnectednessFanos} if we know $\chi(X,\sO_X)=1$ for every weak log del Pezzo surface $(X,\Delta)$. Indeed, granted this, if $f\:Y \to X$ is a (connected) finite \'etale morphism of generic degree $n$, we have that $(Y, f^*\Delta)$ is a weak log del Pezzo surface and so by \cite[18.3.9]{FultonIntersection}
\[
1=\chi(Y,\sO_Y) = n \cdot \chi(X,\sO_X) =n, 
\]
which means that $f$ is trivial and so is $\pi_1^{\mathrm{\acute{e}t}}(X)$.

The authors believe the following claim is well-known to experts yet we sketch a proof here for lack of an adequate reference. We extracted the following argument from \cite{PatakfalviPositiveCharAG}. We are thankful to Zsolt Patakfalvi for letting us include (a sketch of) his argument in this work.

\begin{claim} \label{cla.EulerCharacteristic}
$\chi(X,\sO_X)=1$
\end{claim}
\begin{proof}[Proof of claim]
Of course, this could be thought of as a direct consequence of Kawamata--Viehweg vanishing for weak log del Pezzo surfaces, which is now known to hold for $p > 5$ by \cite[Theorem 1.1]{ArvidssonBernasconiLaciniKVVforLogDelPezzosp>5}. In the general case, we may proceed as follows. 

The key (vanishing) fact is that Grauert--Riemenschneider vanishing holds for log surfaces; see \cite{KollarKovacsBirationalGeometrySurfaces}.\footnote{In other words, if $f\:S' \to S$ is a proper birational morphism between normal surfaces (over an algebraically closed field) and $S'$ is smooth, then $R^1f_* \omega_{S'}=0$.} Thus, we may let $f\:Y \to X$ be a minimal resolution and $g\: Y \to Z$ be a minimal model. Since both $Y$ and $Z$ are smooth they have rational singularities. Hence, the Leray spectral sequence shows that the vanishing of $H^i(X,\sO_X)$ is equivalent to the vanishing of $H^i(Z,\sO_Z)$. Now, define $\Gamma$ on $Y$ by the formula $K_Y+\Gamma = f^*(K_X+\Delta)$. Then, by the negativity lemma, $\Gamma$ is effective, and further $-(K_Y + \Gamma)$ is big and nef. Moreover, the coefficients of $\Gamma$ are $\leq 1$ as $(X,\Delta)$ has KLT singularities. On the other hand, one sees that $-(K_Z+g_*\Gamma)$ is big and nef, and so $-K_Z$ is big. In particular, $Z$ has negative Kodaira dimension and so $Z$ is either a ruled surface or the projective plane. Finally, one rules out the possibility that $Z$ is ruled over a nonrational curve by using that $K_Z+g_*\Gamma$ is big and nef.\footnote{To see this, consider a ruling $Z \to C$ with $C \not\cong \bP^1$. Consider the cases $Z \cong \bP^1 \times C$ and $Z \not\cong \bP^1 \times C$ separately. In the former case, intersect $-(K_Z+g_*\Gamma)$ with general fiber of $f$ to contradict that $-(K_Z+g_*\Gamma)$ is big. In the latter case, intersect $-(K_Z+g_*\Gamma)$ with the exceptional section of $f$ to contradict that $-(K_Z+g_*\Gamma)$ is nef.} That is, $Z$ is either the projective plane or a Hirzebruch surface, and so $H^1(Z,\sO_Z)=0$.

To see that $H^2(Z,\sO_Z)=0$, simply notice that $H^2(Z,\sO_Z)\cong H^0(Z,K_Z)$ and use that $-K_Z$ is big to obtain the desired vanishing.
\end{proof}
Finally, let $U \subset X$ be a big open. Since $(X, \Delta)$ is a KLT \emph{surface} it is $\mathbb{Q}$-Gorenstein \cite[Corollary 4.11]{TanakaMMPExcellentSurfaces} and that $X$ with no boundary is also KLT (\cite[Corollary 2.35]{KollarMori}). Hence, by \cite{HaraDimensionTwo}, we deduce that $X$ is strongly $F$-regular as long as $p > 5$. Then, we may use \cite{BhattCarvajalRojasGrafSchwedeTucker} to conclude the existence of a finite Galois cover $f\: Y \to X$ that is \'etale over $U$ (in particular quasi-\'etale) such that for $V=f^{-1} U$ we have $\pi_1^{\mathrm{\acute{e}t}}(V)=\pi_1^{\mathrm{\acute{e}t}}(Y)$. Nonetheless, we must have that $(Y,f^*\Delta)$ is weak log del Pezzo as $f$ is quasi-\'etale, and so $\pi_1^{\mathrm{\acute{e}t}}(Y)=1$. Hence, $\pi_1^{\mathrm{\acute{e}t}}(U) = \Gal(Y/X)$ is finite.
\end{proof}

\subsubsection{$\tau$-modules} We shall see that tameness can be understood in terms of the action of Frobenius on local cohomology modules so that we need to study modules equipped with a Frobenius action. Following \cite{BockleCohomologicalTheoryOfCrystals}, one has the following definition.
\begin{definition}
Let $X$ be an $\bF_p$-scheme. A \emph{$\tau$-module} on $X$ is a quasi-coherent $\sO_X$-module $\sF$ equipped with an $\sO_X$-linear map $\phi \: \sF \to F_* \sF$. Equivalently, letting $\sO_X[F]$ denote the sheaf of non-commutative rings given by $\sO_X\{F\}/(F\cdot a - a^p \cdot F)$ where $\sO_X\{F\}$ denotes the $\sO_X$-algebra obtained by adding a free non-commutative variable, a $\tau$-module is nothing but a left $\sO_X[F]$-module (that is quasi-coherent as sheaf of $\sO_X$-modules).
\end{definition}

As the reader may expect, the category of coherent $\tau$-modules over a field $\kay$ (of positive characteristic) should resemble the category of finitely generated modules over the PID $\kay[t]$. Indeed, if $\kay$ is perfect, a right division algorithm holds for $\kay[F]$. Consequently, every left ideal of $\kay[F]$ is principal and every submodule of a finite rank free module is free. In particular, a $\tau$-module over $\kay$ is a finite direct sum of cyclic modules $\kay[F]/\kay[F]f$ with $f=0$ or a power of an irreducible polynomial in $\kay[F]$. See \cite[\S14.e]{MilneAlgebraicGroups} for details. In particular, we have the following well-known result (see for instance \cite[III, Lemma 4.13]{MilneEtaleCohomology} and the references therein). 
\begin{theorem} \label{th.KeyDecomposition}
Let $\kay$ be a perfect field of characteristic $p>0$. Let $V$ be a coherent $\tau$-module and $\phi\: V \to F_{*}V$ be the $\kay$-linear map realizing the action of $F$ on $V$ (i.e. $\phi(v)=F\cdot v$). Then, $V$ admits a decomposition $V = V^{\mathrm{nil}} \oplus V^{\mathrm{ss}}$ of $\tau$-modules such that $\phi$ is nilpotent on $V^{\mathrm{nil}}$ and $\phi$ is an isomorphism on $V^{\mathrm{ss}}$. Moreover, if $\kay$ is algebraically closed, then $V^{\mathrm{ss}}$ admits a $\kay$-basis $\{v_1, \ldots ,v_n\}$ such that $\phi(v_i)=v_i$ and further $V^F = \langle v_1, \ldots ,v_n \rangle_{\bF_p}$.
\end{theorem}

In this work, our main examples of $\tau$-modules are going to be cohomology modules. More precisely, let $f \: X \to Y$ be a morphism of $\bF_p$-schemes, and $\sF$ be a $\tau$-module with structural map $\phi\:\sF \to F_* \sF$ on $X$. Since Frobenius is functorial and affine (in particular $\mathbf{R}F_*=F_*$), we see that $\mathbf{R}f_* \phi$ defines a left $\sO_Y[F]$-module structure on $\mathbf{R}f_* \sF$ and so a $\tau$-module structure on its cohomology $R^k f_* \sF$, for
\[
\mathbf{R} f_* \big(F_* \sF\big) =  \mathbf{R} f_* \big(\mathbf{R}F_* \sF\big) = \mathbf{R}F_* \big(\mathbf{R}f_* \sF\big) = F_* \mathbf{R}f_* \sF.
\]
If $g \: Y \to Z$ is a further morphism of $\bF_p$-schemes, then the left $\sO_Z[F]$-module structure of $\mathbf{R}(g \circ f)_* \sF$ is obtained by mapping (pushing forward) the left $\sO_Y[F]$-module structure of $\mathbf{R}f_* \sF$ by $\mathbf{R} g_*$. 

We shall focus our attention on two subsheaves of $\sF$. Namely, $\sF^{\perp F} \coloneqq \ker \phi$ which is an $\sO_X$-submodule of $\sF$, and also $\sF^{F} \coloneqq \ker(\phi-\id)$ which is a subsheaf of $\sF$ as sheaves of $\bF_p$-modules, where $\id \: \sF \to F_* \sF$ is (locally defined) as the additive map $m \mapsto F_*m$. We observe that $(-)^{\perp F}$ defines a left exact functor from the category of $\tau$-modules on $X$ to the category of $\sO_X$-modules, whereas $(-)^F$ defines a left exact functor from $\tau$-modules to $\bF_p$-sheaves on $X$ (sheaves of $\bF_p$-modules). Observe that all the previous remarks apply to sheaves on the \'etale (or flat) topology of $X$.

\begin{theorem}[{\cite[Proposition 10.1.7, Lemma 10.3.1]{BockleCohomologicalTheoryOfCrystals}}] \label{th,.ExactnessOfFrobFixedPoints}
Let $X$ be a noetherian $\bF_p$-scheme. The functor of taking Frobenius fixed elements $(-)^F$ induces an \emph{exact} functor from the category of \'etale $\tau$-modules to the category of \'etale $\mathbb{F}_p$-sheaves. In particular, the canonical morphism
\[
R^k f_* \big(\sF^F\big) \to \big(R^k f_* \sF\big)^F
\]
is an isomorphism for all \'etale $\tau$-modules $\sF$ on $X$.
\end{theorem}
\begin{proof}
For the last statement, notice that by applying $\mathbf{R}f_*$ to the exact sequence
\[
0 \to \sF^F \to \sF \xrightarrow{F-\id} F_*\sF
\]
we obtain a quasi-isomorphism $\mathbf{R}f_*(\sF^F) \to (\mathbf{R}f_*\sF)^F$. The statement then follows by taking cohomologies and using exactness of $(-)^F$ to say that the $k$-th cohomology of $(\mathbf{R}\sF)^F$ is $\big(\mathbf{R}^kf_*\sF\big)^F$.
\end{proof}

\begin{remark}
As pointed out to us by Jakub Witaszek, the use of \autoref{th,.ExactnessOfFrobFixedPoints} can also be replaced by an argument using perfections of schemes. This is formally related to the above theorem due to \cite[Section 1.2]{BhattLurieRiemanHilbert} where it is shown that the category of $\tau$-crystals embeds into the category of perfect Frobenius modules in a way compatible with the Frobenius fixed point functor.
\end{remark}

\subsubsection{Main results concerning tameness}
We are ready to apply the previous results in this section after recalling the following result from \cite{CarvajalFiniteTorsors}, which is well-known to experts.

\begin{lemma} \label{pro.FrefVanishingCohSupport}
Let $X$ be an integral $\bF_p$-scheme and $Z \subset X$ a proper closed subscheme. If $X$ is strongly $F$-regular, then $\sH^k_Z(X,\sO_X)^F=0$ for all $k$. If $X$ is $F$-pure, then $\sH^k_Z(X,\sO_X)^{\perp F}=0$ for all $k$.
\end{lemma}
\begin{proof}
It suffices to consider the affine case. Write $X = \Spec R$ integral and $Z=V(\mathfrak{a})$ such that $\mathfrak{a} = \sqrt{\mathfrak{a}} \neq 0$. If $R$ is $F$-pure, meaning that $F^{\#} \:R \to F_*R$ is split, then $\phi \coloneqq H^k_{\mathfrak{a}}(F^{\#}) \: H^k_{\mathfrak{a}}(R) \to F_* H^k_{\mathfrak{a}}(R)$ is split and so injective, that is $0 =\ker \phi =H^k_{\mathfrak{a}}(R)^{\perp F}$.

Let $a \in H^k_{\mathfrak{a}}(R)$ and choose $r \neq 0$ such that $r \cdot a = 0$. If $R$ is strongly $F$-regular, we find $e \gg 0$ and an $R$-linear map $\vartheta \: F^e_*R \to R$ such that the following composition of $R$-linear maps is the identity
\[
R \xrightarrow{F^{e,\#}} F^e_*R \xrightarrow{\cdot F^e_* r} F^e_* R \xrightarrow{\vartheta} R.
\]
Applying the functor $H^k_{\mathfrak{a}}(-)$, we have that the composition
\[
H^k_{\mathfrak{a}}(R) \xrightarrow{\phi^e} F^e_* H^k_{\mathfrak{a}}(R) \xrightarrow{ F^e_* \cdot r} F^e_* H^k_{\mathfrak{a}}(R) \xrightarrow{\theta} H^k_{\mathfrak{a}}(R)
\]
is the identity as well. However, if $a \in H^k_{\mathfrak{a}}(R)^F$; that is $\phi(a)=F_*a$, then $a$ is mapped to zero along that composition, and so it must be zero.
\end{proof}

\begin{theorem} \label{thm.Tameness}
Work in \autoref{setup} and suppose that $\kay$ is algebraically closed of characteristic $p>5$ and $d=3$. Then, $H^1(S^{\circ}, \sO_{S^{\circ}})^F=H^2_{Z}(S,\sO_S)^F=0$.
\end{theorem}
\begin{proof}
We consider the following diagram
\[
\xymatrix{
S^{\circ} \ar[r]^-{j} \ar[dr]_-{i} & T \ar[d]^-{\pi} \\
& S
}
\]
where $i$ and $j$ are the open immersions whereas $\pi$ is a generalized PLT blowup. Strictly speaking, $\pi$ is the pullback of a PLT blowup $f\: (Y,\Delta_Y) \to U$ (as in \autoref{thm.ExistencePFTRblowUP}) along the canonical morphism $S \to U$. Denote by $(E,\Delta_E)$ the Koll\'ar component of $\pi$, which is a log del Pezzo surface.

Recall that $H^1(S^{\circ}, \sO_{S^{\circ}})$ is the $R$-module given by the global sections of $R^1i_* \sO_{S^{\circ}}$ (as $S = \Spec R$ is affine). On the other hand, by \cite[I, Corollaire 2.11 (29)]{SGA2}, there are canonical isomorphisms of $\tau$-modules on $S$
\[
R^1 i_* \sO_{S^{\circ}} \cong  \sH_Z^2(S,\sO_S), \qquad R^1 j_* \sO_{S^{\circ}} \cong \sH_{\pi^{-1}Z}^2(T,\sO_T).
\] 
In particular, $\big(R^1 j_* \sO_{S^{\circ}}\big)^F=0$ by \autoref{pro.FrefVanishingCohSupport} (and using \autoref{thm.ExistencePFTRblowUP} to ensure $F$-regularity of $T$). On the other hand, $\mathbf{R} i_ * = \mathbf{R}\pi_* \circ \mathbf{R}j_*$, so that $\mathbf{R} i_* \sO_{S^{\circ}} = \mathbf{R}\pi_*\big(\mathbf{R}j_* \sO_{S^{\circ}}\big)$. From the corresponding Grothendieck--Leray spectral sequence, one extracts the following exact sequence of $\tau$-modules on $S$
\[
0 \to R^1\pi_* (j_\ast \sO_{S^\circ}) \to R^1 i_\ast \sO_{S^\circ} \to \pi_\ast \big( R^1 j_\ast \sO_{S\circ}\big).
\]
See \cite[Appendix B, Theorem 1]{MilneEtaleCohomology}. Applying the left exact functor $(-)^F$, we have
\[
0 \to \big( R^1\pi_* (j_\ast \sO_{S^\circ}) \big)^F\to \big( R^1 i_\ast \sO_{S^\circ} \big)^F \to \pi_\ast \big( R^1 j_\ast \sO_{S\circ}\big)^F = 0.
\]
Therefore, it suffices to prove the vanishing
\[
\big( R^1\pi_* (j_\ast \sO_{S^\circ}) \big)^F = 0.
\]
Observe the cohomology sheaf in question is defined on $S$, which is the spectrum of a strictly local ring. In particular, \'etale cohomology on $S$ coincides Zariski cohomology. Carrying out computations on the \'etale site, we have 
\[
\big( R^1\pi_* (j_\ast \sO_{S^\circ}) \big)^F = \big( R^1\pi_* (j_\ast \bG_{\mathrm{a}}) \big)^F =  R^1\pi_* \big(j_\ast \bG_{\mathrm{a}}^F\big) =  R^1\pi_* (j_\ast \bZ/p\bZ) =  R^1 \pi_* (\bZ/p\bZ)
\]
by the Artin--Schreier short exact sequence
\[
0 \to \bZ/p\bZ \to \bG_{\mathrm{a}} \xrightarrow{F-\id} \bG_{\mathrm{a}} \to 0
\] and \autoref{th,.ExactnessOfFrobFixedPoints}.
However,
\[
 R^1\pi_* (\bZ/p\bZ)_{\bar{x}} = H^1(T_{\mathrm{\acute{e}t}},\bZ/p\bZ).
\]
by \cite[III, Theorem 1.15]{MilneEtaleCohomology}. On the other hand, pulling back along the closed embbeding $E \to T$ induces an isomorphism
\[
H^1(T_{\mathrm{\acute{e}t}},\bZ/p\bZ) \xrightarrow{\cong} H^1(E_{\mathrm{\acute{e}t}},\bZ/p\bZ) 
\]
by the proper base change theorem \cite[VI, Corollary 2.7]{MilneEtaleCohomology}. However,
\[
 H^1(E_{\mathrm{\acute{e}t}},\bZ/p\bZ) \cong \Hom_{\mathrm{cont}}\big(\pi_1^{\mathrm{\acute{e}t}}(E), \bZ/p\bZ\big)
\]
as pointed sets; see \cite[III, \S4]{MilneEtaleCohomology}. The result then follows from \autoref{thm.FundamentalGrouoDelPezzo}.
\end{proof}

\begin{remark} \label{rem.RemarkDifferentExplanation}
With notation as in the proof of \autoref{thm.Tameness}, we have more generally that pulling back along $E \to T$ induces an isomorphism $\pi_1^{\mathrm{\acute{e}t}}(E) \xrightarrow{\cong} \pi_1^{\mathrm{\acute{e}t}}(T)$; see \cite[IV, Proposition 2.2]{DeligneBoutotCohomologiETale} \cf \cite{ArtinAlgebraicApproximationOfStructuiresOverCmpleteLocalRIngs}, and so $\pi_1^{\mathrm{\acute{e}t}}(T) = 1$ by \autoref{thm.FundamentalGrouoDelPezzo}. Consider the the short exact sequence
\[0 \to H^0(T,\sO_T)/(F-\id)H^0(T,\sO_T) \to H^1(T_{\mathrm{\acute{e}t}},\bZ/p\bZ) \to H^1(T,\sO_T)^F \to 0,
\]
see \cite[III, Proposition 4.12]{MilneEtaleCohomology}. Notice that $R \to H^0(T,\sO_T)$ is an isomorphism as any global section of $T$ restricts to a global section $T \smallsetminus E \cong S \smallsetminus \{x\}$ which is an element of $R$ (for $R$ is normal and so $\mathbf{S}_2$). Moreover, since $R$ is strictly local, we see that $H^0(T,\sO_T)/(F-\id)H^0(T,\sO_T) = 0$.\footnote{Indeed, nonzero elements of $R/(F-\id)R$ correspond to nontrivial $\bZ/p\bZ$-torsors over $S$ via $r \in R/(F-\id)R \mapsto \Spec \big(R[t]/(t^p-t+r) \big) \to S$, which are trivial if $R$ is strictly local.\label{footnote2}} Thus, we obtain the vanishing
\[
 H^1(T,\sO_T)^F=0
\]
which would have been a consequence of the rationality of $S$.
\end{remark}

\begin{corollary} \label{cor.PGroupTorsors}
With the same setup of \autoref{thm.Tameness}, $H^1(S^{\circ}_{\mathrm{fl}},G) = *$ for all $p$-groups $G$. In particular, for all $p$-groups $G$, every $G$-torsor over $S^{\circ}$ is trivial.
\end{corollary}
\begin{proof}
Just as in \cite[\S4.3]{CarvajalFiniteTorsors}, this is a formal consequence of \autoref{thm.Tameness}. Indeed, say the order of $G$ is $p^e$. We may proceed by induction on $e$. The case $e=1$ follows from \cite[III, Proposition 4.12]{MilneEtaleCohomology} which establishes a short exact sequence
\[
0 \to R/(F-\id)R \to H^1(S^{\circ}_{\mathrm{\acute{e}t}},\bZ/p\bZ) \to H^1(S^{\circ},\sO_{S^{\circ}})^F \to 0
\]
However, $R/(F-\id)R = 0$ as $R$ is strictly local (see footnote~\ref{footnote2} on page~\pageref{footnote2}). By \autoref{thm.Tameness}, $H^1(S^{\circ},\sO_{S^{\circ}})^F=0$. Then, since $\bZ/p\bZ$ is abelian and \'etale, $ H^1(S^{\circ}_{\mathrm{fl}},\bZ/p\bZ)=H^1(S^{\circ}_{\mathrm{\acute{e}t}},\bZ/p\bZ)=0$. For arbitrary $e>1$, use \cite[I, Corollary 6.6]{LangAlgebra}. That is, there is a short exact sequence
\[
1 \to \bZ/p\bZ \to G \to H \to 1
\]
where $H$ is a $p$-group of order $p^{e-1}$. Then, we conclude by taking (non-abelian) cohomology \cite[III, Proposition 4.5]{MilneEtaleCohomology} and the inductive hypothesis.
\end{proof}

For the next corollary, we use the following lemma.

\begin{lemma} \label{lem.FinitegenerationH^2}
Let $(R,\fram)$ be a quasi-excellent noetherian normal local ring of dimension $d \geq 3$, then $H^2_{\fram}(R)$ is a finitely generated $\kay$-module.
\end{lemma}
\begin{proof}
Quite generally, $H^2_{\fram}(R)$ is an artinian $R$-module \cite[Proposition 3.5.4 (a)]{BrunsHerzog}. It then suffices to show $H^2_{\fram}(R)$ is noetherian. To that end, $\hat{R}$ is a quotient of a regular ring by the Cohen structure theorem. However, $H^2_{\hat{\fram}}(\hat{R})= \hat{R} \otimes_R H^2_{\fram}(R)$ and so $H^2_{\fram}(R)$ is finitely generated as an $R$-module if and only if so is $H^2_{\hat{\fram}}(\hat{R})$ as an $\hat{R}$-module; see \cite[\href{https://stacks.math.columbia.edu/tag/03C4}{Tag 03C4}]{stacks-project}. Now observe that $\hat{R}$ is normal \cite[\href{https://stacks.math.columbia.edu/tag/0C23}{Tag 0C23}]{stacks-project}. 
Then one uses \cite[Expos\'e VIII, Corollaire 2.3]{SGA2} which establishes that $H^2_{\fram}(\hat{R})$ is finitely generated if $H^1_{\p R_{\p}}(\hat{R}_{\mathfrak{p}})=0$ for all prime ideals $\mathfrak{p} \subset \hat{R}$ of height $d-1 \geq 2$, which follows from normality.
\end{proof}

\begin{corollary} \label{cor.FrationalityImpliesRational}
With the same setup of \autoref{thm.Tameness}, if $S$ is $F$-injective then it is rational.
\end{corollary}
\begin{proof}
For a KLT singularity, rationality is the same as Cohen--Macaulayness \cite{KovacsRationalSingularities}. Hence, the statement amounts to the vanishing $H^2_{\fram}(R)=0$. However, putting together \autoref{lem.FinitegenerationH^2}, \autoref{th.KeyDecomposition}, and \autoref{thm.Tameness}, it suffices to show $H^2_{\fram}(R)^{\perp F}=0$. This, however, is trivially the case for $F$-injective singularities.
\end{proof}

Next, we study the Frobenius annihilated elements on $H^1(S^{\circ},\sO_{S^{\circ}})$ assuming rationality.

\begin{theorem} \label{thm.UnipotentTameness}
Work in the setup of \autoref{thm.Tameness}. If $S$ is rational, then $H^1(S^{\circ}, \sO_{S^{\circ}})^{\perp F}=H^2_{Z}(S,\sO_S)^{\perp F}=0$.
\end{theorem}
\begin{proof}
The proof commences identically to the one for \autoref{thm.Tameness} up to the point where we need to prove the vanishing
\[
\big( R^1\pi_* (j_\ast \sO_{S^\circ}) \big)^{\perp F} = 0.
\]
Unfortunately, we do not know of any analogous result to \autoref{th,.ExactnessOfFrobFixedPoints} for the functor $(-)^{\perp F}$. Hence, we need to prove such a vanishing by different methods which in our case shall require rationality. However, the following proof is valid for both functors $(-)^F$ and $(-)^{\perp F}$. We proceed as follows, by \cite[I, Corollaire 2.11, (28)]{SGA2}, there is a natural exact sequence
\[
0 \to \sH_{\pi^{-1}Z}^0 (T,\sO_T) \to \sO_T \to j_* \sO_{S^{\circ}} \to \sH_{\pi^{-1}Z}^1(T,\sO_T) \to 0
\]
where we see straight away that $\sH_{\pi^{-1}Z}^0 (T,\sO_T)=0$ as $T$ is integral.
Hence, using the vanishing $R^k \pi_* \sO_T$ for all $k>0$ (see \autoref{thm.RationalityVanishing}), we have that
\[
R^1 \pi_* (j_* \sO_{S^{\circ}}) \cong R^1\pi_* \sH_{\pi^{-1}Z}^1(T,\sO_T) = H^1\big(T,\sH_{\pi^{-1}Z}^1(T,\sO_T)\big)^{\sim}
\]
where the actions of Frobenius are compatible by the naturality of the short exact sequence (i.e. these are isomorphisms of $\tau$-modules). On the other hand, we have a spectral sequence
\[
H^k\big(T,\sH_{\pi^{-1}Z}^l(T,\sO_T)\big) \Rightarrow H^{k+l}_{\pi^{-1}Z}(T,\sO_T),
\]
compatible with Frobenius actions by functoriality, and so it is a spectral sequences of $\tau$-modules. From this spectral sequence, we extract the following exact sequence (of $\tau$-modules) of low degrees (starting with $E_2^{2,0} = 0$, \cf \cite[Appendix B, paragraph before Theorem 1]{MilneEtaleCohomology})
\[
E_2^{2,0} = 0 \to E_1^2 \to E_2^{1,1} \to E_2^{3,0} = H^3\big(T, \sH^0_{\pi^{-1} Z}(T, \sO_T)\big) = 0,
\]
where 
\[
E_1^2 = \ker\Big(H^2_{\pi^{-1}Z}(T, \sO_T) \to H^0\big(T, \sH^2_{\pi^{-1} Z}(T, \sO_T)\big)\Big).
\] 
Since $H^1\big(T,\sH_{\pi^{-1}Z}^1(T,\sO_T)\big) = E_2^{1,1} \cong E_1^2$, it therefore suffices to show that
\[
H_{\pi^{-1}Z}^2(T,\sO_T)^{\perp F}=0.
\]
Write $\pi^{-1}Z=Z'\cup E$ where $Z'$ is the strict transform of $Z$ along $\pi$ (i.e. $Z'$ is the Zariski-closure of $Z \cap (S \smallsetminus \{x\})$ in $T$). We then have the Mayer--Vietoris sequence
\[
\cdots \to H^2_{Z'\cap E}(T,\sO_T) \to H^2_{Z'}(T,\sO_T) \oplus H^2_E(T,\sO_T) \to H^2_{\pi^{-1}Z}(T,\sO_T) \to H^3_{Z'\cap E}(T,\sO_T) \to \cdots
\]
which, by its naturality, is a sequence of $\tau$-modules. Notice that $(Z'\cap E)_{\mathrm{red}}$ is either empty or a finite set of closed points of $T$; say $x_1,\ldots,x_n \in T$.\footnote{Indeed, $Z \cap (S\smallsetminus \{x\})=Z \smallsetminus \{x\}$ is either empty or a finite set of codimension-$2$ points (which correspond to the generic points of the irreducible components of $Z$). Hence, in the latter case, $Z'$ consists of all the specializations of these points in $T$ which must be all closed points in $T$. In other words, $Z'$ meets $E$ only at closed points.\label{footnote}}
In particular,
\[
H^2_{Z'\cap E}(T,\sO_T) = \bigoplus_{i=1}^n H^2_{x_i}(T,\sO_T) = \bigoplus_{i=1}^n H^2_{\fram_{x_i}}(\sO_{T,x_i}) = 0, 
\]
which is zero as $T$ is Cohen--Macaulay and has dimension $3$. Hence, we have an exact sequence
\[
0 \to H^2_{Z'}(T,\sO_T) \oplus H^2_E(T,\sO_T) \to H^2_{\pi^{-1}Z}(T,\sO_T) \to \bigoplus_{i=1}^n H^3_{\fram_{x_i}}(\sO_{T,x_i})
\]
Since $(-)^{\perp F}$ is left exact and $H^3_{\fram_{x_i}}(\sO_{T,x_i})^{\perp F}=0$ by \autoref{pro.FrefVanishingCohSupport}, we obtain
\[
H^2_{\pi^{-1}Z}(T,\sO_T)^{\perp F}= H^2_{Z'}(T,\sO_T)^{\perp F} \oplus H^2_E(T,\sO_T)^{\perp F}.
\]
Next, observe that $Z' \subset T$ is a closed subscheme that passes through only finitely many closed points, namely $x_1, \ldots ,x_n$. Hence, if we denote by $\mathfrak{a}_i \subset \sO_{T,x_i}$ the ideal defining the pullback of $Z'$ along $\Spec \sO_{T,x_i} \to T$, the canonical homomorphism
\[
\rho \: H^2_{Z'}(T,\sO_T) \to  \bigoplus_{i=1}^n H^2_{\mathfrak{a}_i}(\sO_{T,x_i})
\]
is injective. To see this, note first that $\sH_{Z'}^i(T,\sO_T) = 0$ for $i=0,1$ as the complement of $Z'$ in $T$ is a big open (and $T$ is normal). Therefore, $E_2^{2,0} = 0$ and $E_2^{1,1} = 0$ in the spectral sequence
\[
H^k\big(T,\sH_{Z'}^l(T,\sO_T)\big) \Rightarrow H^{k+l}_{Z'}(T,\sO_T).
\]
From this, we deduce that $E_1^2 = \ker \big(H^2_{Z'}(T, \sO_T) \to H^0(T, \sH_{Z'}^2(T, \sO_T) \big) = 0$, i.e. the canonical map
\[
H^2_{Z'}(T,\sO_T) \to H^0\big(T,\sH_{Z'}^2(T,\sO_T)\big)
\]
is injective. Of course, $\rho$ is none other than the composition
\[
H^2_{Z'}(T,\sO_T) \to H^0\big(T,\sH_{Z'}^2(T,\sO_T)\big) \to \bigoplus_{i=1}^n H^2_{\mathfrak{a}_i}(\sO_{T,x_i})
\]
where the arrow to the right is defined via the restriction-to-stalks map and the fact that $\sH_{Z'}^2(T,\sO_T)_{x_i} = H^2_{\mathfrak{a}_i}(\sO_{T,x_i})$ (which holds as local cohomology commutes with localizations \cite[Proposition 7.15]{Twenty-FourHoursOfLocalCohomology}). Since $x_1, \ldots ,x_n$ is the list of closed points supporting the sheaf $\sH_{Z'}^2(T,\sO_T)$, that map is injective and so is $\rho$.

Now, since $\rho$ is compatible with the Frobenius actions, $H^2_{Z'}(T,\sO_T)^{\perp F}$ injects into
\[
\bigoplus_{i=1}^n H^2_{\mathfrak{a}_i}(\sO_{T,x_i})^{\perp F} =0
\]
and so $H^2_{Z'}(T,\sO_T)^{\perp F} = 0$ (the above vanishing is obtained again by using \autoref{pro.FrefVanishingCohSupport} and the strong $F$-regularity of $T$). Thus, we are left with proving the vanishings $H^2_E(T,\sO_T)^{\perp F}=0$. In fact, we have that $H^2_E(T,\sO_T)=0$. To see this, consider the long exact sequence
\[
\cdots \to H^1(T,\sO_T) \to H^1(T\smallsetminus E, \sO_{T \smallsetminus E}) \to H^2_E(T,\sO_T) \to H^2(T,\sO_T) \to \cdots
\]
noticing that $H^1(T,\sO_T)=0 =H^2(T,\sO_T)$ (this is the vanishing $R^i\pi_* \sO_T =0$, $i=1,2$). Hence,
\[
H^2_E(T,\sO_T) \cong H^1(T\smallsetminus E, \sO_{T \smallsetminus E}) \cong H^1\big(S \smallsetminus \{x\}, \sO_{S \smallsetminus \{x\}}\big) \cong H^2_{\fram}(R)=0,
\]
where the last vanishing holds as $S =\Spec R$ is Cohen--Macaulay.
\end{proof}

\begin{corollary} \label{cor.fulltameness}
Working in \autoref{setup}, suppose $\kay$ is algebraically closed of characteristic $p>5$ and $d=3$. If $S$ is rational, then the restriction map
\[
\varrho^1_S(G) \: H^1(S_{\mathrm{fl}}, G) \to H^1(S^{\circ}_{\mathrm{fl}},G)
\]
is surjective for all finite unipotent group-schemes $G/\kay$. In particular, if $Y \to X$ is a $G$-quotient by a finite unipotent group-scheme $G/\kay$ that is a $G$-torsor in codimension-$1$, then $Y \to X$ is a $G$-torsor.
\end{corollary}
\begin{proof}
As in \cite[\S4.3]{CarvajalFiniteTorsors}, the first, local statement is a formal consequence of \autoref{thm.Tameness} and \autoref{thm.UnipotentTameness}. The second, global statement follows because checking whether a given $G$-quotient is a torsor can be done locally in the \'etale topology. Also, recall that being a torsor is an open condition \cite[Corollary 3.13]{CarvajalFiniteTorsors}
\end{proof}

\subsubsection{Questions of interest}

\autoref{cor.fulltameness} raises the following three questions, which should be compared to the ones in \cite[\S 4.7]{CarvajalFiniteTorsors}.
\begin{question} \label{1st.question}
Work in \autoref{setup}. Is the restriction map $H^1(S_{\mathrm{fl}}, G) \to H^1(S^{\circ}_{\mathrm{fl}}, G)$ surjective for all unipotent group-schemes $G/\kay$ if $p$ is large enough?
\end{question}

\begin{question} \label{2nd.question}
Work in \autoref{setup}. Is the restriction map $H^1(S_{\mathrm{fl}}, G) \to H^1(S^{\circ}_{\mathrm{fl}}, G)$ surjective for all unipotent group-schemes $G/\kay$ if $S$ is rational?
\end{question}

\begin{question} \label{3rd.question}
Letting $\kay$ be an algebraically closed field of characteristic $p>5$, is there an action of $\bm{\alpha}_p$ on $\hat{\bA}^3_{\kay}$ whose quotient $\hat{\bA}^3_{\kay} \to \hat{\bA}^3_{\kay}/\bm{\alpha}_p$ is an $\bm{\alpha}_p$-torsor in codimension $1$ and $\hat{\bA}^3_{\kay}/\bm{\alpha}_p$ is KLT?
\end{question}

\begin{remark}
As mentioned in \autoref{rem.Rationality}, \cite{ArvidssonBernasconiLaciniKVVforLogDelPezzosp>5} shows that KLT threefold singularities are rational in characteristics $p>5$. In particular, \autoref{cor.fulltameness} answers \autoref{3rd.question} affirmatively. Of course, it also answers \autoref{1st.question} affirmatively in the $3$-dimensional case by taking $p>5$.
\end{remark}

\section{On the structure of tame covers via local tame fundamental groups} \label{sec.LocalToGlobal}

In this section, we establish an analog of the main results in \cite{BhattCarvajalRojasGrafSchwedeTucker} (\cf \cite{GrebKebekusPeternellEtaleFundamental, StibitzFundamentalGroups}) in the context of tame covers and tame fundamental groups. In \autoref{sec.finiteness}, we will apply this result to the case $X$ is the source of a generalized PLT blowup and $P$ is the corresponding Koll\'ar component. We consider the following setup and notation.

\begin{setup} \label{setup.LocalToGlobal}
Let $X$ be a normal integral scheme, and let $Z \subset X$ be a closed subscheme of codimension at least $2$ with open complement $U=X \smallsetminus Z$. Let $P$ be a prime divisor on $X$---whose restriction to $U$ we denote by $P$ as well---that is itself a normal (integral) scheme. We will denote the ideal sheaf corresponding to $P$ by $\p$. 
Given a geometric point $\bar{x} \to U$, we denote by $P_{\bar{x}}$ the pullback of $P$ to $U_{\bar{x}} \coloneqq \Spec \sO_{U,\bar{x}}^{\mathrm{sh}}$.
\end{setup}

\begin{remark} \label{rem.GlobalToLocalPrimality}
Working in \autoref{setup.LocalToGlobal}, observe that $P_{\bar{x}}$ is a prime divisor on $\sO_{U,\bar{x}}^{\mathrm{sh}}$. Indeed, note that $\sO_{U,\bar{x}}^{\mathrm{sh}}$ is normal (as it is the colimit of \'etale $\sO_{U,x}$-algebras (\cite[\href{https://stacks.math.columbia.edu/tag/033C}{Lemma 033C}]{stacks-project}, \cite[\href{https://stacks.math.columbia.edu/tag/037D}{Lemma 037D}]{stacks-project}) and $\sO_{U,x}$ is normal). Similarly, $P_{\bar{x}}$ is given by the spectrum of the strict henselization of $\sO_{U,x}/\p_x$ (\cite[\href{https://stacks.math.columbia.edu/tag/05WS}{Lemma 05WS}]{stacks-project}) and thus also normal since $\sO_{U,x}/\p_x$ is normal. Since $P_{\bar{x}}$ is the spectrum of a local ring it is integral. 
\end{remark}

\subsubsection*{\'Etale and tame fundamental groups}
For the reader's convenience, we briefly recall the fundamental groups of interest in this paper. For further details, we recommend \cite{MurreLecturesFundamentalGroups,GrothendieckMurreTameFundamentalGroup,CadoretGaloisCategories,CarvajalRojasStaeblerTameFundamentalGroupsPurePairs}. Let $X$ be a normal integral scheme with field of functions $K$, and let $P$ be a prime divisor on $X$. Denoting by $\eta \in X$ the generic point of $X$, we write $\bar{\eta} \to X$ for a chosen geometric generic point, which amounts to the choice of a separable closure $K^{\mathrm{sep}}$ of $K$. We denote by $\mathsf{FEt}(X)$ the Galois category of finite \'etale covers over $X$. Of course, $\mathsf{FEt}(K)$ is nothing but the Galois category (whose minimal/connected objects consist) of finite separable extensions $L/K$, and $\mathsf{FEt}(X)$ can be realized as the full Galois subcategory of $\mathsf{FEt}(K)$ given by those finite separable extensions $L/K$ such that the normalization of $X$ in $L$; say $X^L$, is \'etale over $X$. In particular, we have a canonical surjective homomorphism of topological groups
\[
\pi_1^{\mathrm{\acute{e}t}}(K)=\Gal(K^{\mathrm{sep}}/K) \coloneqq \varprojlim_{\substack{ K^{\mathrm{sep}}/L/K \\ L/K \text{ Galois}}} {\Gal(L/K)} \twoheadrightarrow \pi_1^{\mathrm{\acute{e}t}}(X) = \varprojlim_{\substack{ K^{\mathrm{sep}}/L/K \\ L/K \text{ Galois} \\ X^L/X \in \mathsf{FEt}(X)}} {\Gal(L/K)}
\]
where we have implicitly chosen $\bar{\eta} \to X$ as our common base point. However, we will always drop it from our notation.

Further, we denote by $\mathsf{Rev}^P(X)$ the Galois category of tame covers over $X$ with respect to $P$. More precisely, $\mathsf{Rev}^P(X)$ is the full Galois subcategory of $\mathsf{FEt}(K)$ (whose minimal/connected objects are) given by the finite separable extensions $L/K$ such that the normalization of $X$ in $L$, say $Y \to X$, is such that $Y_{X \smallsetminus P} \to X \smallsetminus P$ is \'etale and $L/K$ is tamely ramified with respect to the DVR $\sO_{X,P} \subset K$. Thus, we consider
\[
\pi_1^{\mathrm{t},P}(X)\coloneqq \varprojlim_{\substack{K^{ \mathrm{sep}}/L/K\\ L/K \text{ Galois}\\ X^L/X \in \mathsf{Rev}^P(X)}} {\Gal(L/K)}
\]
where the limit traverses all finite Galois extensions $L/K$ (inside $K^{\mathrm{sep}}$) such that $X^L \to X$ is \'etale away from $P$ but at worst tamely ramified over the generic point of $P$.

Summing up, we have canonical continuous and surjective homomorphisms
\[
\Gal(K^{\mathrm{sep}}/K) \twoheadrightarrow \pi_1^{\mathrm{\acute{e}t}}(X) \twoheadrightarrow \pi_1^{\mathrm{t},P}(X).
\]

We recall now some terminology from \cite{CarvajalRojasStaeblerTameFundamentalGroupsPurePairs} which will be important in the following discussion. To do so, it is important to notice that given a Galois object $f\: Y \to X$ in $\mathsf{Rev}^{P}(X)$ of (generic) degree $d_f$ we have an equality
\[
d_f=n_f \cdot e_f \cdot i_f
\]
where $n_f$ is the number of irreducible components of $f^{-1} P$,  $e_f$ is the \emph{ramification index}, and $i_f$ is the \emph{inertial degree}. More precisely, if $Q_1, \ldots, Q_{n_f}$ are the irreducible components of $\bigl(f^{-1}(P)\bigr)_{\mathrm{red}}$ and so prime divisors on $Y$, then $e_f$ is the (common) ramification index of the extension of DVRs $\sO_{X,P} \subset \sO_{Y,Q_j}$ for all $j=1,\ldots, n_f$ and $i_f \coloneqq \big[K(Q_j):K(P)\big]$ for all $j=1,\ldots, n_f$.

\begin{terminology}[{\cite[Terminology 3.21]{CarvajalRojasStaeblerTameFundamentalGroupsPurePairs}}]
 Le us recall the following properties on the Galois category $\mathsf{Rev}^{P}(X)$:
\begin{enumerate}
    \item \emph{$P$-irreducibility}: Every connected cover $f\: Y \to X$ in $\mathsf{Rev}^{P}(X)$ satisfies that $Q\coloneqq \bigl(f^{-1}(P)\bigr)_{\mathrm{red}}$ is a prime divisor on $Y$. If $f$ is Galois, this means $n_f = 1$. \emph{In this case, we shall often represent an object in $\mathsf{Rev}^P(X)$ as a pair $(Y,Q)$.}
    \item \emph{Inertial boundedness}: There exists $N \in \bN$ such that $i_f \leq N$ for all Galois objects $f\: Y \to X$ in $\mathsf{Rev}^{P}(X)$.
    \item \emph{Inertial tameness}: The inertial degree $i_f$ is prime-to-$p$ for all Galois objects $f\: Y \to X$ in $\mathsf{Rev}^{P}(X)$.
    \item \emph{Inertial decantation}: Assuming the $P$-irreducibility of $\mathsf{Rev}^{P}(X)$, inertial decantation means that every Galois cover $f\: Y \to X$ in $\mathsf{Rev}^{P}(X)$ dominates a quasi-\'etale Galois cover $Y' \to X$ in $\mathsf{Rev}^{P}(X)$ whose degree is $i_f$. When such degree is $1$, we say that $f$ is \emph{totally ramified} (i.e. $d_f=e_f$).
\end{enumerate}
\end{terminology}

\begin{remark} \label{rem.PiiredImpliesInDecantation}
The \emph{$P$-irreducibility} of $\mathsf{Rev}^{P}(X)$ implies its inertial decantation; see \cite[Lemma 3.23]{CarvajalRojasStaeblerTameFundamentalGroupsPurePairs}.
\end{remark}

\begin{theorem} \label{thm.FundamentalTheorem}
Work in \autoref{setup.LocalToGlobal} and suppose that:
\begin{enumerate}[label=(\arabic*)]
    \item Every connected cover $f\: V \to U$ in $\mathsf{Rev}^P(U)$ is such that $f^{-1}P$ is connected,
    \item $\mathsf{Rev}^{P_{\bar{x}}}(U_{\bar{x}})$ is $P_{\bar{x}}$-irreducible and inertially bounded for all geometric points $\bar{x} \to U$.
\end{enumerate}
 Then, the Galois category $\mathsf{Rev}^P(U)$ has the three following properties:
 \begin{enumerate}[label=(\Alph*)]
     \item $P$-irreducibility.
     \item Inertial decantation.
     \item There exists a quasi-\'etale Galois cover $(\Bar{U},\Bar{P})$ in $\mathsf{Rev}^P(U)$ such that every quasi-\'etale Galois cover in $\mathsf{Rev}^{\Bar{P}}(\Bar{U})$ is \'etale.
 \end{enumerate}
 Furthermore, suppose that $\bar{P}$ is normal and $\pi_1^{\mathrm{\acute{e}t}}(\bar{P} \cap \bar{U})$ is finite. Then, there exists an \'etale cover $(\tilde{U}, \Tilde{P})$ in $\mathsf{Rev}^{\bar{P}}(\Bar{U})$ (with $\Tilde{P}$ normal) such that for all Galois covers $(V,Q)$ in $\mathsf{Rev}^{\Tilde{P}}\bigl(\Tilde{U}\bigr)$ we have:
     \begin{enumerate}[label=(\roman*)]
    \item $(V,Q)/(\tilde{U}, \Tilde{P})$ is totally ramified,  
    \item $Q \to \tilde{P}$ is an isomorphism and so $Q$ is normal,
    \item $(V,Q)/(\tilde{U}, \Tilde{P})$ is \'etale-locally a Kummer cover over $\tilde{U}_{\mathrm{reg}}$. That is, given a geometric point $\bar{x} \to \Tilde{U}_{\mathrm{reg}}$, there exists an \'etale neighborhood $W \to \Tilde{U}$ of $\bar{x}$ such that $V_{W} \to W$ is isomorphic in $\mathsf{Rev}^{P_W}(W)$ to $W[T]/(T^{e}-w) \to W$, where $\Div w = P_{W}$ for some $0 \neq w \in \Gamma(W,\sO_W)$, and $e=[K(V):K(\tilde{U})]$ is the ramification index of $(V,Q)/(\tilde{U}, \Tilde{P})$.
\end{enumerate}
\end{theorem}

\begin{remark}\label{rem.FunRem}
Property (A) in \autoref{thm.FundamentalTheorem} is truly essential. In its pressence, we can think of connected objects in $\mathsf{Rev}^P(U)$ as pairs $(V,Q)$ where $V/U$ is a tamely ramified cover with respect to $P$ and $Q$ is the only prime divisor on $V$ lying over $P$. In this way, $V\smallsetminus Q \to U \smallsetminus P$ is \'etale and $\sO_{U,P} \to \sO_{V,Q}$ is a tamely ramified extension of DVRs, so that $[K(V):K(U)]=e \cdot [K(Q):K(P)]$, where $e$ is the ramification index. Consequently, the Galois category consisting of the objects of $\mathsf{Rev}^P(U)$ that lie over a given (connected) object $(V,Q)$ is the same as the Galois category $\mathsf{Rev}^Q(V)$. In particular, given a Galois object $(V,Q)$ in $\mathsf{Rev}^P(U)$, we have a short exact sequence of topological groups
\[
1 \to \pi_1^{\mathrm{t},Q}(V) \to  \pi_1^{\mathrm{t},P}(U) \to \Gal(V/U) \to 1. 
\]
\end{remark}

Before getting into the proof of \autoref{thm.FundamentalTheorem}, we consider instructive to have the following preliminary discussion on stratifications and Abhyankar's lemma.

\subsection{Stratifications}
Let $X$ be a normal integral scheme. Suppose that we are given a tower of quasi-\'etale Galois covers
\[
X = X_0 \xleftarrow{\gamma_0} X_1 \xleftarrow{\gamma_1} X_2 \xleftarrow{\gamma_2} \cdots
\]
such that $\gamma_{i,\bar{x}}$ is \'etale for sufficiently large $i$ for every geometric point $\bar{x} \to X$. More precisely, suppose that for all geometric points $\bar{x} \to X$ there is $N_{\bar{x}} \in \bN$ such that $\gamma_{i,\bar{x}}$ is \'etale for all $i \geq N_{\bar{x}}$. It is natural to ask whether this implies the existence of a uniform constant $N \in \bN$ (independent of the geometric point) such that $\gamma_i$ is \'etale for all $i \geq N$. This question was considered in the works \cite{GrebKebekusPeternellEtaleFundamental,BhattCarvajalRojasGrafSchwedeTucker}, where positive answers were given over the complex numbers and in positive characteristic, respectively. In both cases, a type of stratification was required. In \cite{GrebKebekusPeternellEtaleFundamental}, the authors employed the so-called Whitney's stratification, whereas in \cite{BhattCarvajalRojasGrafSchwedeTucker} the authors used a result of O.~Gabber \cite[Th\'eor\`eme 1.1, 1.3]{TravauxdeGabber} to construct an analogous stratification. It is worth noting that C.~Stibitz \cite{StibitzFundamentalGroups} has recently introduced a characteristic free stratification; based on de Jong's alterations \cite{deJongAlterations}, that is suitable to answer this question positively.

\begin{theorem} \label{cor.LocalToGlogalQuasiEtaleCovers}
Let $X$ be a normal integral excellent scheme. Let
\[
X = X_0 \xleftarrow{\gamma_0} X_1 \xleftarrow{\gamma_1} X_2 \xleftarrow{\gamma_2} \cdots
\]
be a tower of quasi-\'etale Galois covers over $X$. Suppose that for every geometric point $\bar{x} \to X$ there is $N_{\bar{x}} \in \bN$ such that $\gamma_{i,\bar{x}}$ is \'etale (hence trivial) for all $i \geq N_{\bar{x}}$. Then, there is $N \in \bN$ such that $\gamma_i$ is \'etale for all $i \geq N$.
\end{theorem}
\begin{proof}
This follows straight away from \cite{BhattCarvajalRojasGrafSchwedeTucker}; \cf \cite[Remark 4.2]{BhattCarvajalRojasGrafSchwedeTucker} to see why bounding the order of the local fundamental groups in \cite[Proposition 3.2]{BhattCarvajalRojasGrafSchwedeTucker} is superfluous. Note that, if $X_i$ admits a regular alteration for all large enough $i$, another proof can be found in \cite[\S4, proof of $(i) \Rightarrow (iii)$]{StibitzFundamentalGroups}.
\end{proof}

\subsection{Abhyankar's lemma}

Next, we come to the \'etale-local description of normal totally ramified Galois covers in \autoref{thm.FundamentalTheorem}, which we refer to as Abhyankar's lemma, \cf \cite[\href{https://stacks.math.columbia.edu/tag/0EYG}{Tag 0EYG}]{stacks-project}. The following result is based on a generalization of \cite[\href{https://stacks.math.columbia.edu/tag/0EYG}{Tag 0EYG}]{stacks-project} to singular divisors given in \cite{CarvajalRojasStaeblerTameFundamentalGroupsPurePairs}. 

\begin{theorem} \label{thm.InversionAbhyankar}
Work in the same setup of \autoref{thm.FundamentalTheorem}. Let $f\: (V,Q) \to (U,P)$ be a totally ramified Galois cover in $\mathsf{Rev}^{P}(U)$ of degree $e$. Then, for a given geometric point $\bar{x} \to X$ with $x\in P \cap U_{\mathrm{reg}}$, there exists an \'etale neighborhood $W \to U$ of $\bar{x}$ such that $V_W \to W $ is 
isomorphic in $\mathsf{Rev}^{P_W}(W)$ to $W[T]/(T^{e}-w) \to W$, where $\Div w = P_W$ for some $0 \neq w \in \Gamma(W,\sO_W)$.
\end{theorem}

\begin{proof}
We use the notation of \autoref{thm.FundamentalTheorem} and \autoref{setup.LocalToGlobal}. Consider the cartesian diagram
\begin{equation} \label{eqn.StrictLocalizationGalois}
\xymatrix{
V \ar[d]_-{f} & V_{\bar{x}} \ar[l] \ar[d]^-{f_{\bar{x}}} \\
U & U_{\bar{x}} \ar[l]
}
\end{equation}
where $f_{\bar{x}}$ is an object in $\mathsf{Rev}^{P_{\bar{x}}}(U_{\bar{x}})$; see \cite[Lemma 2.2.8]{GrothendieckMurreTameFundamentalGroup}.\footnote{Note that $V_{\bar{x}}$ is normal as $V_{\bar{x}} \to V$ is pro-\'etale and $V$ is normal.} Since $f\: Q \to P$ is a finite birational (as $f\: V \to U$ is totally ramified) morphism and $P$ is normal, we obtain that $f\: Q \to P$ is an isomorphism (this is a simple case of Zariski's main theorem). Therefore, since $x\in P$, there is only one point of $V_{\bar{x}}$ lying over $\bar{x}$. Hence, $f_{\bar{x}}$ is connected and totally ramified. In particular, $V_{\bar{x}}$ is a connected normal scheme, which has the following consequence:

\begin{claim}
In \autoref{eqn.StrictLocalizationGalois}, $f_{\bar{x}}$ is Galois. 
\end{claim}
\begin{proof}[Proof of claim]
Let $W \to U$ be a connected \'etale neighborhood of $\bar{x} \to U$, and consider the diagram
\begin{equation*}
\xymatrix{
V \ar[d]_-{f} & V_W \ar[d]^-{f_W} \ar[l] & V_{\bar{x}} \ar[l] \ar[d]^-{f_{\bar{x}}} \\
U & W \ar[l] & U_{\bar{x}} \coloneqq \Spec \sO_{U,\bar{x}}^{\mathrm{sh}} \ar[l]
}
\end{equation*}
where both inner squares are cartesian and the outer square is none other than \autoref{eqn.StrictLocalizationGalois}. In particular, $V_W$ is a normal integral scheme.

Now, by localizing the left cartesian diagram at the generic point of $U$, we conclude that $K(V_W) = K(V) \otimes_{K(U)} K(W)$, this because all the four arrows in this diagram are quasi-finite morphisms in between irreducible schemes. 
In other words, $K(V)$ and $K(W)$ are subfields of $K(W)$ whose compositum is $K(V_W)$. In particular,  $K(V_W)/K(W)$ is Galois by \cite[Chapter VI, Theorem 1.12]{LangAlgebra}, which is to say $f_W$ is Galois as an object of $\mathsf{Rev}^{P_W}(W)$.

To prove $K(V_{\bar{x}})/K(U_{\bar{x}})$ is Galois, we notice that separability is clear as \'etaleness is preserved under base change. We must explain normality of $K(V_{\bar{x}})/K(U_{\bar{x}})$. To this end, we consider $0 \neq b \in K(V_{\bar{x}})$ and let $\alpha(t) = t^n +a_{n-1}t^{n-1}+\cdots + a_1t+a_0 \in K(U_{\bar{x}})[t]$ be its minimal polynomial. By clearing denominators, let $0 \neq a \in \sO_{U,\bar{x}}^{\mathrm{sh}}$ be such that $a \cdot  \alpha(t)$ has coefficients in $\sO_{U,\bar{x}}^{\mathrm{sh}}$.
Moreover, we write $b=b'/b''$ where $0 \neq  b', b'' \in \Gamma(V_{\bar{x}}, \sO_{V_{\bar{x}}}) = f_*\sO_{V} \otimes_{\sO_U} \sO_{U,\bar{x}}^{\mathrm{sh}}$ (where this is an equality of $\sO_{U,\bar{x}}^{\mathrm{sh}}$-algebras). Recall that $\sO_{U,\bar{x}}^{\mathrm{sh}} = \varinjlim_{W} \Gamma(W,\sO_W)$ where the colimit traverses all \'etale neighborhoods $W \to U$ of $\bar{x} \to U$. In particular, we may write \[
\Gamma(V_{\bar{x}}, \sO_{V_{\bar{x}}}) = f_*\sO_{V} \otimes_{\sO_U} \sO_{U,\bar{x}}^{\mathrm{sh}} = \varinjlim_W \big( f_* \sO_V \otimes_{\sO_U} \Gamma(W,\sO_W) \big) = \varinjlim_W \Gamma(V_W,\sO_{V_W}),
\]
where $f_*\sO_V \otimes_{\sO_U} \Gamma(W,\sO_W) = \Gamma(V_W,\sO_{V_W})$ is an equality of $\Gamma(W,\sO_W)$-algebras.
Let $W \to U$ be a sufficiently small connected \'etale neighborhood of $\bar{x} \to U$ such that $\Gamma(W, \sO_W) \ni a, aa_{n-1}, \ldots, aa_0$,\footnote{To be precise, we choose $W$ such that the image of $\Gamma(W,\sO_W)$ under the canonical homomorphism $\Gamma(W,\sO_W) \to \sO_{U,\bar{x}}^{\mathrm{sh}}$ contains $a$ and $aa_i$ for all $i$.} and such that $b',b'' \in f_*\sO_V \otimes_{\sO_U} \Gamma(W,\sO_W) = \Gamma(V_W,\sO_{V_W})$. In particular, we have that $a_i \in K(W)$ for all $i=1,\ldots,n-1$, and further $b\in K(V_W)$. However, we proved above that $K(V_W)/K(W)$ is Galois (so normal). Hence, the polynomial $\alpha(t)$ must split in $K(V_W)$ and consequenctly in $K(V_{\bar{x}})$. 
\end{proof}

By assumption $x \in U_{\mathrm{reg}}$ and so $U_{\bar{x}} \coloneqq \Spec \sO_{U,\bar{x}}^{\mathrm{sh}}$ is a regular strictly henselian scheme and in particular factorial. Hence, $P_{\bar{x}} = \Div w$ for some $0 \neq w \in U_{\bar{x}} \coloneqq \Spec \sO_{U,\bar{x}}^{\mathrm{sh}}$. Moreover, by purity of the branch locus for regular schemes, it follows that every Galois quasi-\'etale cover over $U_{\bar{x}} \coloneqq \Spec \sO_{U,\bar{x}}^{\mathrm{sh}}$ is trivial. Putting everything together, we conclude that $\pi_1^{\mathrm{t},P_{\bar{x}}}(U_{\bar{x}})= \hat{\Z}^{(p)}$ by \cite[Theorem 3.29]{CarvajalRojasStaeblerTameFundamentalGroupsPurePairs}. Hence, $V_{\bar{x}} \to U_{\bar{x}}$ is further cyclic and so a Kummer cover with ramification index $e$, \cf \cite[Lemma 3.34]{CarvajalRojasStaeblerTameFundamentalGroupsPurePairs}. This description remains valid on a sufficiently small \'etale neighborhood $W$ of $\bar{x} \to U$ where $\Div w = P_W$.
\end{proof}

The following result will be useful in \autoref{sec.finiteness}.
\begin{corollary} \label{cor.UseAbhyankarLemma}
With the same hypothesis as in \autoref{thm.InversionAbhyankar}, set $U' \coloneqq U_{\mathrm{reg}} \smallsetminus P$. Additionally, suppose that $P$ contains all closed points of $U$. Then,  $V_{U'} \to {U'}$ is a connected $\bm{\mu}_{e}$-torsor in the sense of \cite[III, \S4]{MilneEtaleCohomology} (i.e. a torsor in the fppf or flat topology).
\end{corollary}
\begin{proof}
Since $P$ contains all closed points of $U$, $P \cap U_{\mathrm{reg}}$ contains all the closed points of 
$U_{\mathrm{reg}}$.\footnote{Indeed, a closed point of $U_{\mathrm{reg}}$ is necessarily a closed point of $U$ by \cite[\href{https://stacks.math.columbia.edu/tag/07P8}{Tag 07P8}]{stacks-project} (or directly by \cite[\href{https://stacks.math.columbia.edu/tag/0541}{Tag 0541}]{stacks-project} and the fact that the Zariski closure of a point is irreducible).} Applying \autoref{thm.InversionAbhyankar}, we find an \'etale covering $\{W_i \to U_\mathrm{reg}\}_{i \in I}$ such that $f_{W_i}\: V_{W_i} \to W_i$ is a Kummer cover of degree $e$. In particular, the pullback of $f_{W_i}$ to $U'$ is a $\bm{\mu}_e$-torsor. Since these torsors are all induced by a morphism $f\: V \to U$ they clearly glue. Finally, we observe that connectedness follows from normality.
\end{proof}

\begin{remark} \label{rem.Z/etorsorsCyclicCovers}
Let $Y \to X$ be a connected $\bm{\mu}_e$-torsor (in the flat topology) with $X$, $Y$ normal integral schemes.
We may use Kummer theory to describe $Y \to X$ in terms of cyclic covers; see \cite[III, \S4]{MilneEtaleCohomology}. Indeed, we have that there exists an invertible sheaf $\sL$ on $X$ of index $e$ such that $Y \to X$ is isomorphic to a cyclic cover $\Spec_X \bigoplus_{i=0}^{e-1} \sL^i \to X$. In particular, if the Picard group of $X$ is finitely generated, $e$ cannot be arbitrarily large.
\end{remark}

\subsection{Proof of \autoref{thm.FundamentalTheorem}}
With the above in place, we are ready to proceed with the proof of the main theorem of this section. 
\begin{proof}[Proof of \autoref{thm.FundamentalTheorem}]
First of all, we recall that $P_{\bar{x}}$ is a prime divisor on $\sO_{X,\bar{x}}^{\mathrm{sh}}$ so that our local hypothesis (2) makes sense; see \autoref{rem.GlobalToLocalPrimality}.

\begin{proof}[Proof of property (A)] 
Let $f\: V \to U$ be a connected cover in $\mathsf{Rev}^P(U)$. Since $f$ is finite and surjective, proving (A) amounts to showing that there is exactly one DVR of $K(V)$ lying over $\sO_{U,  P}\subset K(U)$; equivalently, that there is exactly one point of $V$ lying over the generic point of $P$. Indeed, $(f^{-1}P)_{\mathrm{red}}$ is the reduced closed subscheme of $V$ whose irreducible components are given by the prime divisors associated to the (necessarily) codimension-$1$ points of $V$ lying over the generic point of $P$. Since we assume $f^{-1} P$ to be connected by hypothesis (1), so is $(f^{-1}P)_{\mathrm{red}}$. Hence, it suffices to prove that the prime divisors of $V$ lying over $P$ do not intersect. Suppose for the sake of contradiction that two of such divisors, say $Q$ and $Q'$, meet at a closed point $y \in V$ and let $x= f(y) \in Y$. Let $W \to U_{\bar{x}}$ be the connected component of  $f_{\bar{x}}\:V_{\bar{x}} \to U_{\bar{x}}$ that contains $\bar{y} \to V_{\bar{x}}$; see \cite[Lemma 3.31]{CarvajalRojasStaeblerTameFundamentalGroupsPurePairs}. Then, $W \to U_{\bar{x}}$ is a connected object in  $\mathsf{Rev}^{P_{\bar{x}}}(U_{\bar{x}})$ that violates $P_{\bar{x}}$-irreducibility and so our hypothesis (2); a contradiction. Indeed, the pullback of $Q$ and $Q'$ to $W$ are two prime divisors on $W$ lying over $P_{\bar{x}}$.
\end{proof}

\begin{proof}[Proof of property (B)] Property (B) is a formal consequence of (A); see \autoref{rem.PiiredImpliesInDecantation}. 
\end{proof}

\begin{proof}[Proof of property (C)]
The existence of the stated quasi-\'etale Galois cover in $\mathsf{Rev}^P(U)$ follows from applying \autoref{cor.LocalToGlogalQuasiEtaleCovers}. Indeed, consider a tower
\[
U = U_0 \xleftarrow{\gamma_0} U_1 \xleftarrow{\gamma_1} U_2 \xleftarrow{\gamma_2} \cdots
\]
of quasi-\'etale Galois covers in $\mathsf{Rev}^P(U)$.
Since $\mathsf{Rev}^{P_{\bar{x}}}(U_{\bar{x}})$ inertially bounded for all geometric points $\bar{x} \to U$, the maps $\gamma_{i,\bar{x}}$ are trivial for sufficiently large $i$. By applying \autoref{cor.LocalToGlogalQuasiEtaleCovers}, we conclude that the maps $\gamma_i$ are \'etale for sufficiently large $i$. By the abstract nonsense regarding Galois categories (keeping in mind \autoref{rem.FunRem}), we obtain the desired statement.
\end{proof}

\begin{proof}[Proof of the final statement]
Consider a tower
\[
(\Bar{U}, \Bar{P}) = (\Bar{U}_0, \Bar{Q}_0) \xleftarrow{\beta_0} (\bar{U}_1, \Bar{Q}_1) \xleftarrow{\beta_1} (\Bar{U}_2, \Bar{Q}_2) \xleftarrow{\beta_2} \cdots
\]
of \'etale Galois covers in $\mathsf{Rev}^{\Bar{P}}(\Bar{U})$. When restricted to $\Bar{P}$, we obtain a tower of \'etale covers
\[
\Bar{P} = \Bar{Q}_0 \xleftarrow{\alpha_0} \bar{Q}_1 \xleftarrow{\alpha_1} \Bar{Q}_2 \xleftarrow{\alpha_2} \cdots
\]
Here, we use that the maps $\beta_i$ are \'etale to say that $\Bar{Q}_{i+1} = (\beta_i^{-1}\bar{Q}_i)_{\mathrm{red}}=\beta_i^{-1}\bar{Q}_i$. In particular, we have that all the prime divisors $\Bar{Q}_i$ are normal. In this manner, the $\alpha_i$ must eventually be all isomorphisms as $\pi_1^{\mathrm{\acute{e}t}}(\Bar{P})$ is assumed finite. Using this, \autoref{rem.FunRem}, and property (B), we obtain an object $(\tilde{U}, \tilde{P})$ that satisfies property (i). We obtain that  $(\tilde{U}, \tilde{P})$ satisfies property (ii) and (iii) as a direct application of \autoref{thm.InversionAbhyankar}.
\end{proof}
This completes the proof of \autoref{thm.FundamentalTheorem}.
\end{proof}

\section{Finiteness and tameness of local \'etale fundamental groups} \label{sec.finiteness}

We now come to the finiteness of $\pi_1^{\mathrm{\acute{e}t}}(S^{\circ})^{(p)}$ and eventually to the tameness and finiteness of $\pi_1^{\mathrm{\acute{e}t}}(S^{\circ})$. \emph{Throughout this section, we work in \autoref{setup} assuming $d = 3$ and $\kay$ to be algebraically closed of characteristic $p>5$}.

First, use \autoref{thm.ExistencePFTRblowUP} to ensure the existence of a \emph{generalized} PLT blowup $\pi\: (T,\Delta_T) \to S$ extracting a Koll\'ar component $(E, \Delta_E)$,\footnote{To be precise, we construct the PLT blowup on a Zariski open neighborhood of $x\in X$ and then we pull it back to $S$. However, as pointed out to us by an anonymous referee, this may now be done directly by using the recent work \cite{BhattMaPatakfalviSchwedeTuckerWaldronWitaszekG+RegMMPthrefoldMixChar}.} which is a log del Pezzo pair. Recall that $(T, E+\Delta_T)$ is a PLT pair whereas $(T,E)$ is a purely $F$-regular one. Define $\Delta_S \coloneqq \pi_* \Delta_T$, so $(S,\Delta_S)$ is KLT (\cf \autoref{rem.logfano}). We denote by $Z' \subset T$ the strict transform of $Z$ along $\pi$,\footnote{That is, $Z'$ is the Zariski-closure of $Z \cap (S \smallsetminus \{x\})$ in $T$. Note that, by construction, $S \smallsetminus \{x\}$ is an open subscheme of both $S$ and $T$.} and we write $U \coloneqq T \smallsetminus Z'$. By abuse of notation, we denote by $E$ the restriction of $E$ to $U$. Further, we observe that $S^{\circ} = U \smallsetminus E$. The following commutative diagram depicts the situation
\begin{equation} \label{eqn.setup}
\xymatrix{
S^{\circ} \ar[r]^-{\iota} \ar[rd]_-{i} & U \ar[d]^-{\varpi} \ar[r]^-{o} & T \ar[ld]^-{\pi} \\
& S
}
\end{equation}
where $i$, $\iota$, and $o$ are the open immersions. Next, we consider the following observation of Grothendieck--Murre whose proof we revisit for the reader's convenience and for the sake of completeness (as we work in a slightly different setup).
 \begin{proposition}[{\cite[Corollary 9.8]{GrothendieckMurreTameFundamentalGroup}}] \label{cla.GrotehndieckMurre}
The restriction functor defined by $\iota$
\[
\mathsf{Rev}^E(U) \to \mathsf{FEt}(S^{\circ})
\]
induces a surjective homomorphism of topological groups
\[
\lambda \:\pi_1^{\mathrm{\acute{e}t}}(S^\circ) \to \pi_1^{\mathrm{t},E}(U)
\]
that induces an isomorphism on prime-to-$p$ parts:
\[
\mu = \lambda^{(p)} \: \pi_1^{\mathrm{\acute{e}t}}(S^\circ)^{(p)} \xrightarrow{\cong} \pi_1^{\mathrm{t},E}(U)^{(p)}.
\]
Thus, we have surjective homomorphisms of topological groups
\[
\pi_1^{\mathrm{\acute{e}t}}(S^\circ) \twoheadrightarrow \pi_1^{\mathrm{t},E}(U) \twoheadrightarrow \pi_1^{\mathrm{\acute{e}t}}(S^\circ)^{(p)}
\]
\end{proposition}
\begin{proof}
By the abstract nonsense regarding Galois categories \cite[\S5.2.1]{MurreLecturesFundamentalGroups}, to show $\lambda$ is surjective, we must prove that any connected (Galois) cover $V \to U$ in $\mathsf{Rev}^E(U)$ remains connected when restricted to $S^{\circ}$. However, this follows from normality as connectedness and irreducibility are equivalent notions under normality and because irreducibility is preserved under restriction to open subschemes.\footnote{In general, any non-empty open subset of an irreducible topological space is (dense and) irreducible \cite[I, \S1, Excercise 1.6]{Hartshorne}.}

To see why $\mu = \lambda^{(p)}$ is injective, we proceed as follows. We must verify that any Galois cover $Y \to S^{\circ}$ in $\mathsf{FEt}(S^{\circ})$ of order prime-to-$p$ is obtained as the restriction of some Galois cover in $\mathsf{Rev}^E(U)$ (whose order is necessarily prime-to-$p$). To this end, we consider the Galois field extension $K \subset K(Y)$. Since $\pi$ is birational, we have $K(U) = K$. Thus, we may let $V \to U$ be the normalization of $U$ in $K(Y)$. Of course, the pullback of $V \to U$ along $\iota$ is $Y \to S^{\circ}$, and moreover $K(Y)/K$ is tame with respect to $\sO_{U,E}$ as $p$ does not divide $\big[K(Y):K\big]$.
\end{proof}

\begin{remark} \label{rem.GettingThingsClear}
In the proof of \autoref{cla.GrotehndieckMurre}, we mentioned the relationship between the Galois categories $\mathsf{Rev}^E(U)$, $\mathsf{FEt}(S^{\circ})$. We would like expand on this as it will facilitate and clarify the upcoming discussion. Let $K_1/K$ be a finite Galois extension. We may then take the normalization of the diagram \autoref{eqn.setup} so that we have:
\[
\xymatrix{ 
& & S_1^{\circ} \ar@{.>}[dll] \ar[r]^-{\iota_1} \ar[rd]_-{i_1} & U_1 \ar@{.>}[dll] \ar[d]^-{\varpi_1} \ar[r]^-{o_1} & T_1 \ar@{.>}[dll] \ar[ld]^-{\pi_1} &  \\
S^{\circ}  \ar[r]^-{\iota} \ar[rd]_-{i} & U  \ar[d]^-{\varpi} \ar[r]^-{o} & T \ar[ld]^-{\pi} &  S_1 \ar@{.>}[dll] \\
& S 
}
\]
where the dotted, unlabeled arrows represent the normalization morphisms. Note that both top (slanted) rectangles are cartesian as well as the one most to the left, for $i$, $\iota$, and $o$ are open immersions.  Notice that $S_1 = \Spec R_1$ where $R_1$---the integral closure of $R$ in $K_1$---is a strictly local ring as $R$ is strictly local and $S_1$ is connected. We may say that the content of \autoref{cla.GrotehndieckMurre} is that $S_1^{\circ}/S^{\circ}$ is \'etale if $U_1/U$ is at worst tamely ramified over $E$ and the converse (trivially) holds provided that $[K_1:K]$ is prime-to-$p$. Consider next the commutative square \autoref{keydiagram} making the lateral right face of the diagram. This is the diagram we shall focus on:
\begin{align}
\label{keydiagram}
\xymatrix{
T \ar[d]_-{\pi} & T_1 \ar[d]^-{\pi_1}  \ar[l]_-{g} \\
S  & S_1 \ar[l]_-{f}
}
\end{align}
Note that $(S_1, f^* \Delta_S\eqqcolon \Delta_{S_1})$ is a KLT singularity (\cite[2.43]{KollarSingulaitieofMMP}). Similarly, we define $\Delta_{T_1} \coloneqq g^* \Delta_T$. Since $g$ and $f$ are finite (and so proper), we have that $\pi_1$ is proper for $\pi$ is proper. Thus, $\pi_1$ is a proper birational morphism with exceptional locus equal to $E_1 \coloneqq g^{-1} E$ by the commutativity of the diagram. We conclude that $E_1$ is then connected by Zariski's main theorem; see \cite[III, Corollary 11.4]{Hartshorne}.
\end{remark}

With the above in place, our next goal is to prove that $\pi_1^{\mathrm{t},E}(U)$ is finite, so that $\pi_1^{\mathrm{\acute{e}t}}(S^{\circ})^{(p)}$ is finite as well in the rational case.

\begin{theorem}[{\cf \cite[Theorem 3.4]{XuZhangNonvanishing}, \cite[Proposition 5.2]{HaconWitaszekMMPLowCharacteristic}}] \label{theo.mainresultbody}
The groups $\pi_1^{\mathrm{t},E}(U)$ and so $\pi_1^{\mathrm{\acute{e}t}}(S^{\circ})^{(p)}$ are finite.
\end{theorem}
\begin{proof} 
We apply \autoref{thm.FundamentalTheorem} to the case $X \leftrightarrow T$, $Z \leftrightarrow Z'$, $U \leftrightarrow U$, and $P \leftrightarrow E$.
First, we verify the hypothesis (1) and (2) of \autoref{thm.FundamentalTheorem} to hold in this context.

\begin{claim} \label{cla.CheckingHypothesis}
$\mathsf{Rev}^E(U)$ satisfies hypothesis (1) and (2) in  \autoref{thm.FundamentalTheorem}. Furthermore, for any connected object $(U_1,E_1)$ in $\mathsf{Rev}^E(U)$, we have:
\begin{itemize}
    \item $E_1$ is as normal prime divisor on $T_1$,
    \item $(T_1,E_1)$ is a purely $F$-regular pair and so $\pi_1 \: (T_1, \Delta_{T_1}) \to S_1$ is a generalized PLT blowup with respective Koll\'ar component $(E_1, \Delta_{E_1}$),
    \item $\pi_1^{\mathrm{\acute{e}t}}(E_1 \cap U_1)$ is finite.
\end{itemize}
\end{claim}
\begin{proof}[Proof of claim]
Why $\mathsf{Rev}^E(U)$ satisfies hypothesis (1) was explained in the last paragraph of \autoref{rem.GettingThingsClear}.

We proceed to verify (2). Recall that $(T,E)$ has purely $F$-regular singularities. In particular, for every geometric point $\bar{x} \to U$, the local pair $(U_{\bar{x}}, E_{\bar{x}})$ is a purely $F$-regular local pair in the sense of \cite[Definition 2.3]{CarvajalRojasStaeblerTameFundamentalGroupsPurePairs}. Thus, we may apply \cite[Theorem 4.12]{CarvajalRojasStaeblerTameFundamentalGroupsPurePairs}, which establishes that $\mathsf{Rev}^{E_{\bar{x}}}(U_{\bar{x}})$ is $E_{\bar{x}}$-irreducible, inertially bounded, and inertially tame. 

Let us explain next why $(U_1,E_1)$ is purely $F$-regular, so that $E_1$ is strongly $F$-regular and so normal. Further, $E_1$ is a normal prime divisor as it is connected. The pure $F$-regularity of $(U_1,E_1)$ is a local condition that can be checked at every geometric point. Moreover, it can be checked at every geometric point $\bar{x} \to U$. Since $(U_{\bar{x}}, E_{\bar{x}})$ is a purely $F$-regular, we may apply \cite[Theorem 4.7]{CarvajalRojasStaeblerTameFundamentalGroupsPurePairs} to conclude that every connected component of $(E_{1,\bar{x}},E_{1,\bar{x}})$ is a purely $F$-regular local pair; as desired.

Since $g$ is a tame cover with respect to $E$, $(T_1, E_1 + \Delta_{T_1})$ is PLT by \cite[2.43]{KollarSingulaitieofMMP} and in fact $g^\ast(K_T + E + \Delta_T) = K_{T_1} + E_1 + \Delta_{T_1}$. Further, we see that $-(K_{T_1}+E_1 + \Delta_{T_1})$ is ample over $S$ and the contraction of $E_1$ is precisely $\pi_1$ and $\pi_{1,*}\Delta_{T_1} = \Delta_{S_1}$. This concludes the proof of the second bullet point.

For the finiteness of $\pi_1^{\mathrm{\acute{e}t}}(E_1 \cap U_1)$, we use \autoref{thm.FundamentalGrouoDelPezzo} noting that $E_1 \cap U_1$ is a big open of $E_1$; see footnote~\ref{footnote} on page~\pageref{footnote}.
\end{proof}

By \autoref{cla.CheckingHypothesis}, we may apply \autoref{thm.FundamentalTheorem} to obtain a Galois quasi-\'etale cover $f\: \tilde{U} \to U$ satisfying properties (i) and (ii) in \autoref{thm.FundamentalTheorem}. In particular, it suffices to explain why $\pi_1^{\mathrm{t},\tilde{E}}(\tilde{U})$ is finite. To this end, we use \autoref{cor.UseAbhyankarLemma} and \autoref{dlt.loc.pic.prop}. First, we observe that $\tilde{E}$ goes through every closed point of $\tilde{U}$. Indeed, using the construction and notation of \autoref{rem.GettingThingsClear}, we observe that $\tilde{\pi} \: \tilde{T} \to \tilde{S}$ is a proper morphism over the spectrum of a local ring and $\tilde{E}$ is the fiber over \emph{the} closed point and so it must contain every closed point of $\tilde{T}$. Thus, by \autoref{cor.UseAbhyankarLemma} \cf \autoref{rem.Z/etorsorsCyclicCovers}, the degree of a Galois cover in $\mathsf{Rev}^{\tilde{E}}(\tilde{U})$ is bounded by the degree of a connected prime-to-$p$ cyclic cover over 
\[
\tilde{U}_{\mathrm{reg}} \smallsetminus \tilde{E} = \tilde{S}^{\circ} \cap \tilde{U}_{\mathrm{reg}} = \tilde{S}^{\circ}_{\mathrm{reg}} = \tilde{S}^{\circ} \cap \tilde{S}_{\mathrm{reg}} \eqqcolon O,
\]
and so by the index of a prime-to-$p$ torsion element of $\Pic O$. However, $O$ is a regular big open of $\tilde{S}$ and so $\Pic O = \Cl O = \Cl \tilde{S}$, which has finitely generated by prime-to-$p$ part by  \autoref{dlt.loc.pic.prop}. Thus, the degree of a Galois object in $\mathsf{Rev}^{\tilde{E}}(\tilde{U})$ cannot be arbitrarily large. This implies the existence of a universal Galois cover in $\mathsf{Rev}^{\tilde{E}}(\tilde{U})$ and further the finiteness of $\pi_1^{\mathrm{t},\tilde{E}}(\tilde{U})$.
\end{proof}

The following is the main result of this paper, which we shall prove as a formal Galois-theoretic consequence of \autoref{cor.PGroupTorsors} and \autoref{theo.mainresultbody}.

\begin{corollary} \label{thm.TamenessFundGroup}
The canonical surjection $\pi_1^{\mathrm{\acute{e}t}}(S^{\circ}) \twoheadrightarrow \pi_1^{\mathrm{t},E}(U) \twoheadrightarrow \pi_1^{\mathrm{\acute{e}t}}(S^{\circ})^{(p)}$ is an isomorphism. Hence, $\pi_1^{\mathrm{\acute{e}t}}(S^{\circ})$ is tame and finite.
\end{corollary}
\begin{proof}
Let $K_1/K$ be a Galois extension whose Galois group $G = \Gal(K_1/K)$ has order divisible by $p$. Suppose, for the sake of contradiction, that $K_1/K$ is an object in $\mathsf{FEt}(S^{\circ})$. Formally, this is tantamount to the existence of a homomorphism $\pi_1^{\mathrm{\acute{e}t}}(S^{\circ}) \twoheadrightarrow G$. Let $S_1 \to S$ be the corresponding quasi-\'etale cover, i.e. $S_1 = \Spec R_1$ where $R_1$ is the integral closure of $R$ in $K_1$ which is a strictly local ring. Let $H$ be the $p$-Sylow group of $G$, which is nontrivial by assumption. Let $S_0 \coloneqq S_1/H = \Spec R_1^H$ be the corresponding quotient ($R_1^H$ is also a strictly local ring). Consider the induced factorization $S_1 \to S_0 \to S$ of $S_1 \to S$ noting that $S_0 \to S$ is \'etale over $S^{\circ}$ (and so quasi-\'etale) and $S_1 \to S_0$ is an $H$-torsor over the inverse image of $S^{\circ}$ along $S_0 \to S$, say $S_0^{\circ}$. In particular, $S_0$ is a KLT singularity by \cite[Corollary 2.43]{KollarSingulaitieofMMP} just as $S$ is. Then, applying \autoref{cor.PGroupTorsors} to $S_0$ yields the sought contradiction. Indeed, the pullback of $S_1 \to S_0$ to $S_0^{\circ}$ defines a nontrivial element of $H^1(S_{0,\mathrm{fl}}^{\circ}, H)$---which is a singleton according to \autoref{cor.PGroupTorsors}. 
\end{proof}

\begin{remark}
Note that \autoref{thm.TamenessFundGroup} and so \autoref{thm.Tameness} do not hold in characteristics $p=2,3,5$ by examples in \cite{YasudaDiscrepanciesPcyclicQuotientVarieties,TotaroFailureKVforFanos}. 
\end{remark}

We finalize with the following---by now standard---application \cite{GrebKebekusPeternellEtaleFundamental,BhattCarvajalRojasGrafSchwedeTucker,StibitzFundamentalGroups}. Roughly speaking, it establishes that threefolds with at worst KLT singularities \emph{almost} satisfy Zariski--Nagata--Auslander purity of the branch locus.

\begin{corollary} \label{cor.DirectCorollary}
Work in \autoref{setup} assuming $d=3$ and $\kay$ to be algebraically closed of characteristic $p>5$. Then, every quasi-\'etale cover over $X$ whose degree is a power of $p$ is \'etale. Moreover, there exists a quasi-\'etale Galois cover $\tilde{X} \to X$ with prime-to-$p$ degree such that every quasi-\'etale cover over $\tilde{X}$ is \'etale. Moreover, $\tilde{X} \to X$ can be taken so that it is \'etale over any prescribed big open of $X$.
\end{corollary}
\begin{proof}
The first statement was already noticed in more generality in \autoref{cor.fulltameness}. If $Y \to X$ is a quasi-\'etale cover whose degree is a power of $p$, then the same is true for $Y \times_ X \Spec \sO_{X,\bar{x}}^{\mathrm{sh}} \to \Spec \sO_{X,\bar{x}}^{\mathrm{sh}}$ at every geometric closed point $\bar{x} \to X$. Hence they are all \'etale by \autoref{cor.PGroupTorsors}, and so is $Y \to X$. 

The existence of a quasi-\'etale Galois cover $X' \to X$ such that quasi-\'etale covers over $X'$ are \'etale follows from the methods in \cite{BhattCarvajalRojasGrafSchwedeTucker}, as well as \cite[Theorem 1]{StibitzFundamentalGroups} for $X$ is a variety. However, it also follows by the same methods that there exists a quasi-\'etale Galois cover $\tilde{X} \to X$ with prime-to-$p$ degree such that every quasi-\'etale cover over $\tilde{X}$ with prime-to-$p$ degree is \'etale. We claim that all quasi-\'etale covers over $\tilde{X}$ are \'etale. It suffices to prove this for the connected and further Galois ones. Thus, let $Y \to \tilde{X}$ be a quasi-\'etale Galois cover. Let $H \subset \Gal(Y/\tilde{X})$ be the corresponding $p$-Sylow group, which we may assume nontrivial as otherwise there is nothing to prove. Let $Y \to Y/H \to \tilde{X}$ be the corresponding quotient. Thus, $Y \to Y/H$ is a quasi-\'etale Galois cover with Galois group $H$ and $Y/H \to \tilde{X}$ is quasi-\'etale and so \'etale. In particular, $Y/H$ is a KLT threefold just as $X$ and $\tilde{X}$ are. Hence, by the first part of the corollary, $Y \to Y/H$ must be \'etale and so is $Y \to \tilde{X}$.
\end{proof}

\begin{corollary} \label{cor.FinitenessAndTamenessOfFundGroupsOfBigOpens}
Work in \autoref{setup} assuming $d=3$, $\kay$ to be algebraically closed of characteristic $p>5$, and $(X,\Delta)$ is a log Fano pair. If $U \subset X$ is a big open subset then $\pi_1^{\mathrm{\acute{e}t}}(U)$ is finite and tame. 
\end{corollary}
\begin{proof}
By \autoref{cor.DirectCorollary}, there is a tame quasi-\'etale Galois cover $f\:\tilde{X} \to X$ that is \'etale over $U$ and such that the induced homomorphism $\pi_1^{\mathrm{\acute{e}t}}(\tilde{X}_U) \to \pi_1^{\mathrm{\acute{e}t}}(\tilde{X})$ is an isomorphism. In particular, it suffices to show that $\pi_1^{\mathrm{\acute{e}t}}(\tilde{X})$ is finite and tame. To this end, we observe that $(\tilde{X}, f^* \Delta)$ is a log Fano pair and so rationally chain connected by \cite[Theorem 1.5]{GongyoNakamuraTanakaRationalPointsFanoThreefoldsFiniteFields}. Then, the result follows from \cite{ChambertLoirFundamentalGroupsofRCCVarieties}.
\end{proof}

\appendix

\section{Local Picard schemes and local divisor class groups} \label{sec.LocalPicardSchemes}
Inspired by \cite[II (a)]{MumfordFundGroup}, Grothendieck proposed in \cite[pp. 189-193]{SGA2} the construction of a local Picard scheme; see \cite[Remark 9.4.19]{FGAExplained} for an account. Such a construction would provide a connection between the local cohomology module $H^2_{\mathfrak{m}_{\bar{x}}}(\sO_{X,\bar{x}}^{\mathrm{sh}})$ and the Picard group of the punctured spectrum of $\sO_{X,\bar{x}}^{\mathrm{sh}}$. Also see \cite[\S 3]{EsnaultViehwegSurfaceSingularitiesDominatedSmoothVarieties}, where the $2$-dimensional case is worked out. The construction of a local Picard scheme was obtained by J.~Lipman \cite{LipmanLocalPicardScheme} and subsequently by J.-F.~Boutot \cite{BoutotSchemaPicardLocal}. For the reader's convenience, we briefly recall Boutot's approach and main results. For details, see \cite{BoutotSchemaPicardLocal}.

\emph{Throughout this section, we let $(R,\fram, \kay)$ be a noetherian strictly local $\kay$-algebra and denote by $S$ and $S^{\circ}$ its spectrum and punctured spectrum; respectively.}

For a given $\kay$-algebra $A$, one defines $ R \otimes_{\kay}^{\mathrm{sh}} A$ to be the henselization of $R \otimes_{\kay} A$ along the ideal $\fram \otimes_{\kay} A$. There is a canonical morphism
\[
\alpha_A \: \Spec \bigl( R \otimes_{\kay}^{\mathrm{sh}} A \bigr) \to S,
\]
and one defines
\[
\tilde{S}^{\circ}_A \coloneqq \alpha_A^{-1} \big(S^{\circ} \big).
\]
Next, one considers the presheaf of abelian groups on $\kay$-algebras given by
\[
A \mapsto \Pic \tilde{S}^{\circ}_A.
\]
One defines the \emph{local Picard functor} $\ssPic_{S/\kay}$ to be the sheafification of this presheaf in the \'etale topology on $\kay$-algebras. 

The first main result we need from \cite{BoutotSchemaPicardLocal} is the following:
\begin{theorem}[\cite{BoutotSchemaPicardLocal}] \label{thm.LocalPicardScheme}
Suppose that $\depth R \geq 2$ and $H^2_{\fram}(R)=H^1(S^{\circ}, \sO_{S^{\circ}})$ is finite over $\kay$. Then, $\ssPic_{S/\kay}$ is represented by a (commutative) group-scheme locally of finite type over $\kay$, say $\Pic_{S/\kay}$. Moreover, the tangent space of $\Pic_{S/\kay}$ at the identity is isomorphic to $H^1(S^{\circ}, \sO_{S^{\circ}})$ as $\kay$-modules.
\end{theorem}

\begin{definition}
With notation as in \autoref{thm.LocalPicardScheme}, one refers to $\Pic_{S/\kay}$ as the \emph{local Picard scheme} of $R$ (or $S$). One denotes by $\Pic^0_{S/\kay}$
the connected component of $\Pic_{S/\kay}$ at the identity, which by general principles is an algebraic group over $\kay$; see \cite[Lemma 9.5.1]{FGAExplained} or \cite[Theorem 2.4.1]{BrionStructureTheoremsAlgebraicGroups}. The \emph{local N\'eron--Severi} group is defined as the quotient
\[
\NS_{S/\kay} \coloneqq \Pic_{S/\kay}\bigm/\Pic^0_{S/\kay}.
\]
\end{definition}

\begin{remark} \label{rem.ExistenceCmCase}
With notation as in \autoref{thm.LocalPicardScheme}, suppose that $R$ is quasi-excellent and normal of dimension $d \geq 3$. Then, $\Pic_{S/\kay}$ exists by \autoref{lem.FinitegenerationH^2}.
In that case, the $\kay$-points of $\Pic_{S/\kay}$ correspond to the Picard group of $S^{\circ}$. Precisely,
\[
\Pic_{S/\kay}\big(\Spec(\kay)\big) = \ssPic_{S/\kay} (\kay) = \Pic S^{\circ} \subset \Cl S^{\circ} = \Cl S.
\]
More generally, if $A$ is a strictly local $\kay$-algebra, then the $A$-points on $\Pic_{S/\kay}$ can be computed as follows
\[
\Pic_{S/\kay}\big(\Spec A\big) = \ssPic_{S/\kay}(A) = \Pic \tilde{S}^{\circ}_A
\]
Technically speaking, this is to say that the canonical homomorphism $ \Pic \tilde{S}^{\circ}_A \to \ssPic_{S/\kay}(A)$ is an isomorphism if $A$ is a strictly local $\kay$-algebra; see \cite[II, Proposition 3.2]{BoutotSchemaPicardLocal}. For instance, recall that the tangent space of $\Pic_{S/\kay}$ at the identity is
\[
\mathrm{Lie}(\Pic_{S/\kay}) \coloneqq \ker\Big(\ssPic_{S/\kay}\big(\kay[\varepsilon]/\varepsilon^2\big) \to \ssPic_{S/\kay}(\kay)\Big)
\]
Hence, since $\kay$ and $\kay[\varepsilon]/\varepsilon^2$ are both strictly local $\kay$-algebras, there is a canonical isomorphism
\[
\ker\Big(\Pic\big(S^{\circ}[\varepsilon]/\varepsilon^2\big) \to \Pic(S^{\circ})\Big) \xrightarrow{\cong} \mathrm{Lie}(\Pic_{S/\kay})
\]
and further
\[
\ker\Big(H^1\big(S^{\circ}[\varepsilon]/\varepsilon^2, \sO_{S^{\circ}[\varepsilon]/\varepsilon^2}^{\times} \big) \to H^1(S^{\circ}, \sO_{S^{\circ}}^{\times})\Big) \xrightarrow{\cong} \mathrm{Lie}(\Pic_{S/\kay})
\]
Now, considering the short exact sequence of sheaves on $S^{\circ}[\varepsilon]/\varepsilon^2$
\[
0 \to \sO_{S^{\circ}} \xrightarrow{a \mapsto 1+a \varepsilon}  \sO_{S^{\circ}[\varepsilon]/\varepsilon^2}^{\times} \to \sO_{S^{\circ}}^{\times} \to 1
\]
and the corresponding long exact sequence on cohomology, we see why $\mathrm{Lie}(\Pic_{S/\kay})$ is canonically isomorphic to $H^1(S^{\circ}, \sO_{S^{\circ}}) = H^2_{\fram}(R)$ (\cf \cite[III, Exercise 4.6]{Hartshorne}). For more details, see \cite[II, Proposition 4.2]{BoutotSchemaPicardLocal} whose proof specializes to the above argument by taking $A'\to A \to A_0$ in ibid. to be $\kay[\varepsilon]/\varepsilon^2 \to \kay \to \kay$.
In particular, if $R$ is further Cohen--Macaulay then $\Pic^0_{S/\kay}=0$ as its tangent space at the identity $H^2_{\fram}(R)$ is zero. That is, $\NS_{S/\kay} \cong \Pic_{S/\kay}$.
\end{remark}

\begin{theorem}\cite{BoutotSchemaPicardLocal} \label{thm.FinitenessOfneronSeveri}
Suppose that $\kay$ is algebraically closed, and $R$ is excellent, normal, and has dimension $3$. Then, $\NS_{S/\kay}(\kay)$ is a finitely generated abelian group. 
\end{theorem}
\begin{remark}
Boutot proved \autoref{thm.FinitenessOfneronSeveri} to hold in higher dimensions as well if $R$ is further \emph{strongly desingularizable}; see \cite[V, \S3]{BoutotSchemaPicardLocal}. However, as Boutot remarks, if $\kay$ has characteristic $0$ this condition follows from \cite{HironakaResolution} or else by \cite{AbhyankarResolutionOfSingularities,CossartPiltantResolutionSingularitiesArithmetical3folds} if $\dim R \leq 3$.
\end{remark}

As a direct consequence of \autoref{thm.FinitenessOfneronSeveri}, one obtains the following important corollaries, \cf \cite[V, Corollaire 4.9]{BoutotSchemaPicardLocal}.
\begin{corollary} \label{cor.FiniteGenerationCL}
Work in the setup of \autoref{thm.FinitenessOfneronSeveri}. If $S$ is Cohen--Macaulay, then $ \Pic S^{\circ}$ is a finitely generated abelian group.
\end{corollary}

\begin{proof}
As explained in \autoref{rem.ExistenceCmCase}, $\NS_{S/\kay} \cong \Pic_{S/\kay}$ if $R$ is Cohen--Macaulay. Then, by computing $\kay$-points under this isomorphism, we obtain the claimed statement by using \autoref{thm.FinitenessOfneronSeveri}.
\end{proof}

We are thankful to Thomas Polstra for his help with the following corollary, which answers \cite[Question 1]{PolstraATheoremMCM} in dimension $3$.

\begin{corollary} \label{cor.BoundingTorsionn}
Work in the setup of \autoref{thm.FinitenessOfneronSeveri}. If $S$ is rational, then $\Cl S$ is a finitely generated abelian group.
\end{corollary}
\begin{proof}
Since rational threefold singularities are Cohen--Macaulay, $\Pic S^{\circ}$ is finitely generated by \autoref{cor.FiniteGenerationCL}. Note that the punctured spectrum of a rational threefold singularity is $\mathbb{Q}$-factorial; see \cite{LipmanRationalSingularities,EsnaultViehwegSurfaceSingularitiesDominatedSmoothVarieties}. Indeed, $S^{\circ}$ is a $2$-dimensional normal scheme with rational singularities, and so it has isolated $2$-dimensional rational singularities which all have finite (and so torsion) divisor class group \cite[Lemma 3.1]{EsnaultViehwegSurfaceSingularitiesDominatedSmoothVarieties}. In particular, there is $0< n \in \mathbb{N}$ such that $n \cdot \Cl S \subset \Pic S^{\circ}$ (note that $n$ can be taken to be $1$ if $S$ has an isolated singularity meaning that $S^{\circ}$ is regular).  Thus, there is an exact sequence of abelian groups
\[
0 \to \Pic S^{\circ} \to \Cl S \to T \to 0
\]
where $T$ is $n$-torsion. This immediately implies that torsion in $\Cl S$ is bounded. However, under closer inspection, $T$ is finite and hence $\Cl S$ is finitely generated as it is the extension of finitely generated groups. 
To see why $T$ is finite, consider the following. Let $s_1,\ldots, s_k \in S^{\circ}$ be the finitely many singular (rational) points of $S^{\circ}$. Set $S_i \coloneqq \Spec \sO_{S^{\circ}, s_i}$. As we mentioned before, we know that the groups $\Cl S_i = \Pic (S_i \smallsetminus s_i)$ are finite by \cite[Lemma 3.1]{EsnaultViehwegSurfaceSingularitiesDominatedSmoothVarieties}. Note that $\delta \in \Cl S = \Cl S^{\circ}$ belongs to $\Pic S^{\circ}$ if and only if $\delta$ belongs to the kernel of each homomorphism $\Cl S \to \Cl S_i$. In other words,
\[
\Pic S^{\circ} = \ker\bigg(\Cl S \to \bigoplus_{i=1}^k \Cl S_i\bigg)
\]
which implies that $T$ is isomorphic to a subgroup of the finite group $\bigoplus_{i=1}^k \Cl S_i$.
\end{proof}

\clearpage

\section{On local class groups}
\pagestyle{fancy}
\label{sec.Appendix}

\smallskip
\begin{center}by \textsc{J\'anos Koll\'ar}\end{center}
\medskip

The aim of this appendix is to prove 
two propositions  about  factoriality properties of log terminal schemes.

From now on we work with   Noetherian and excellent schemes.

\begin{proposition} \label{dlt.loc.pic.prop}
Let $(y,Y, \Delta)$ be a 3-dimensional,  dlt, local scheme (see Remark~\ref{future.cite}) of residue characteristic $p>0$.  Then the prime-to-$p$ parts of  $\Cl(Y), \Cl(Y^{\mathrm{h}})$ and of   $\Cl(\hat Y)$  are
 finitely generated.  
\end{proposition}

Here $\Cl(\phantom{Y})$ denotes the class group, $Y^{\mathrm h}$ the  Henselisation  and  $\hat Y$ the completion in the topology induced by the maximal ideal.

The {\it $p^\infty$-part} of an Abelian group $A$ is the subgroup 
$A[p^\infty]\subset A$ consisting of elements whose order is a power of $p$.
The quotient  $A/A[p^\infty]$ is the {\it prime-to-$p$ part} of $A$. 
If $A/A[p^\infty]$ is  finitely generated  then $A\cong A[p^\infty]\oplus A/A[p^\infty]$.

\begin{remark}\label{future.cite} The current references  \cite{KollarrelativeMMPwoQfactoriality, Witaszekflipslowchar} cover the cases when $X$ is essentially of  finite type over a perfect field. Arbitrary excellent 3-folds  (including mixed characteristic) are treated in \cite{h-w-p, BhattMaPatakfalviSchwedeTuckerWaldronWitaszekG+RegMMPthrefoldMixChar}.
\end{remark}

\begin{remark} \label{cl.inj.rem}
Let $(y\in Y)$ be a normal,  local domain with Henselisation $Y^{\mathrm h}$ and completion  $\hat Y$. There are injections
$$
\Cl(Y)\into \Cl(Y^{\mathrm h})\into \Cl(\hat Y);
$$
cf.\ \cite[{\S}1C]{MR0453732}. Thus   finite generation for  $\Cl(\hat Y)$ implies finite generation for
$\Cl(Y)$ and $\Cl(Y^{\mathrm h})$.

Frequently $\Cl(Y^{\mathrm h})$ is much bigger than $\Cl(Y)$
(see Example~\ref{cone.exmp} in dimension 3 or \cite[27]{KollarMapsBetweenLocalPicardGroups} for a 2-dimensional  log canonical example in characteristic 0), but $\Cl(Y^{\mathrm h})= \Cl(\hat Y)$ if $Y$ has an isolated singularity \cite[3.10]{ArtinAlgebraicApproximationOfStructuiresOverCmpleteLocalRIngs}.
\end{remark}

Note that Proposition~\ref{dlt.loc.pic.prop} is new only when $(y,Y)$ is not (or not known to be) rational, which  happens only if $\Char k(y)\in \{2,3,5\}$; see  \cite[Cor.1.3]{ArvidssonBernasconiLaciniKVVforLogDelPezzosp>5}.

\begin{example} \label{cone.exmp}
Let $S$ be a normal, del~Pezzo surface over a field $k$ and $L$ an ample line bundle on $S$ such that  $H^1(S, L^m)= 0$ for $m\geq 2$. 
Set $Y\coloneqq\Spec_k \oplus_{m\geq 0} H^0(S, L^m)$, the affine cone over $(S, L)$.  
Then  $\Cl(Y)= \Cl(S)/[L] $   and there is an exact sequence
$$
0\to H^1(S, L)\to \Cl(\hat Y)\to \Cl(S)/[L]\to 0.
$$
If $H^1(S, L)=0$ then $ \Cl(\hat Y)= \Cl(S)/[L]$; this holds if the characteristic is 0 or $\geq 7$.  However, in characteristics $2,3,5$ there are 
examples with $H^1(S, L)\neq 0$; see \cite{CasciniTanakaPLTThreefoldsNonNormalCenters, BernasconiKVFailsLogDelPezzo, ArvidssonBernasconiLaciniKVVforLogDelPezzosp>5}. For these  
 the $p^\infty$-part of $\Cl(\hat Y)$ contains  $H^1(S, L)$, which is infinite if
$k$  is.
\end{example}

\begin{remark}[Local class group and local Picard group] \label{loc.pic.cl.comp.rem}
Let $(y,Y)$ be a 3-dimensional  singularity. Let
$\eta_i\in \Sing Y\setminus \{y\}$ be the generic points.
There is an exact sequence
$$
0\to \picl(y, Y)\to \Cl(Y)\to \bigoplus_i \Cl(Y_{\eta_i}).
$$
Here the $ Y_{\eta_i}$ are 2-dimensional.
If $(y,Y, \Delta)$ is klt then the  $ Y_{\eta_i}$ are   rational, hence
the  $\Cl(Y_{\eta_i})$ are finite  by
\cite{LipmanRationalSingularities}.
Thus for 3-dimensional, klt singularities, 
$\picl(y, Y)\subset  \Cl(Y)$ is a finite index subgroup.

For equicharacteristic local rings \cite{BoutotGroupePicardLocalAnneauHenselien, BoutotSchemaPicardLocal} defines a 
{\it local Picard scheme}  $\picls(y, Y)$.
 Proposition~\ref{dlt.loc.pic.prop} is equivalent to saying that 
 its connected component  is a unipotent algebraic group. 
 
\end{remark}

It is reasonable to expect that the following generalization of Proposition~\ref{dlt.loc.pic.prop}   holds in all dimensions.

\begin{conjecture} \label{dlt.loc.pic.prop.conj}
Let $(y\in Y, \Delta)$ be a complete, dlt, local scheme  and  $p\coloneqq\Char k(y)$. 
Then 
\begin{enumerate}[label=(\arabic*)]
\item $\Cl(Y)/\Cl(Y)[p^\infty]$ is  finitely generated, and
\item $\Cl(Y)[p^\infty]$ has bounded exponent.
\end{enumerate}
\end{conjecture}

As part of the proof of Proposition~\ref{dlt.loc.pic.prop} we establish the following result about the 
formal-local class groups. 
This seems to be a special case  of  more general finiteness phenomena, which will be considered elsewhere; see  \cite{BoissiereGAbberSermanVarietesLocalmentFactorielles} for another example.

\begin{proposition} \label{lipman.extend.prop}
 Let $X$ be a normal  scheme.  Assume that its codimension 2 singular points are  rational.  Then there is a closed subset $W\subset X$ of codimension $\geq 3$ such that  
\begin{enumerate}[label=(\arabic*)]
\item  $\Cl(X\setminus W)/\pic(X\setminus W)$ is finite and
\item   $X\setminus W$ is formal-locally  $\mathbb{Q}$-factorial. Moreover, the orders of the formal-local class groups   $\Cl(\hat X_x)$ are uniformly bounded for $x\in X\setminus W$.
\end{enumerate}
\end{proposition}

\subsection*{Proof of Proposition~\ref{lipman.extend.prop}}{\ }

Let $x\in X$ be a codimension 2 singular point. After localizing at $x$ we obtain a rational surface singularity $(s,S)$.  
Then $(s,S)$ is $\mathbb{Q}$-factorial by \cite{LipmanRationalSingularities}. 
Thus if $D\subset X$ is any Weil divisor, then $D$ is $\mathbb{Q}$-Cartier at $x$, hence also in an open neighborhood of $x$. This shows that every 
 Weil divisor $D$ is $\mathbb{Q}$-Cartier outside a  closed subset $W(D)\subset X$ of codimension $\geq 3$. We need to show that $W(D)$ is essentially independent of $D$.  This is not obvious. However, we claim that one can follow the steps of Lipman's proof to get  Proposition~\ref{lipman.extend.prop}. To see this, let us first recall the surface case.

\begin{say}[Summary of Lipman's proof]\label{lipman.prop}
 I mostly follow the notation as in
\cite[10.7--10]{KollarSingulaitieofMMP}. 

Let $f\:T\to S$ be a resolution of singularities with exceptional curves $\{C_i:i\in I\}$.
We need the following claims.
\begin{enumerate}[label=(\arabic*)]
\item There are finite field extensions $F_i/k(s)$ such that  $C_i$ is either $\bP^1_{F_i}$ (we call these lines) or a regular conic in $\bP^2_{F_i}$.
\item There is  a sequence of effective, $f$-exceptional cycles
$Z_1\leq \cdots\leq Z_m$ and a map  $\tau\: \{1,\dots, m\}\to I$ such that
\begin{enumerate}
\item  $\sO_T(-Z_m)$ is $f$-ample, 
\item   $Z_1=C_{\tau(1)}$, 
\item $C_{\tau(i+1)}$ is a line for $i\geq 1$ and
we have exact sequences
$$
0\to \sO_{C_{\tau(i+1)}}(-1)\to \sO_{Z_{i+1}}\to \sO_{Z_{i}}\to 0.
$$
\end{enumerate}
\end{enumerate}
The $\mathbb{Q}$-factoriality is now established in 3 steps. The most important is the following.
\medskip

{\it Claim \ref{lipman.prop}.3.}   Let $L$ be any line bundle on $T$ that  has degree 0 on each  $C_i$. Then $L$ is the pull-back of a line bundle from $S$.
Equivalently, there is an exact sequence
$$
\pic(S)\stackrel{f^*}{\longrightarrow} \pic(T)\stackrel{c_1}{\longrightarrow} \bigoplus_{i\in I}\mathbb{Z}[C_i],
$$
where
$c_1(M)\coloneqq\bigl(\deg(M|_{C_i}):i\in I\bigr)$.

\medskip

Proof. We extend the  sequence of effective, $f$-exceptional cycles
$Z_1\leq \cdots\leq Z_m$ by setting
$Z_{am+i}\coloneqq aZ_m+Z_i$.  Then the sequences
(\ref{lipman.prop}.2.c) become
$$
0\to \sO_{C_{\tau(i+1)}}(-1)\otimes \sO_T(-aZ_m)\to \sO_{Z_{am+i+1}}\to \sO_{Z_{am+i}}\to 0.
$$
Since $Z_m$ is  $f$-ample, 
$$
H^1\bigl(C_{\tau(i)}, \sO_{C_{\tau(i+1)}}(-1)\otimes \sO_T(-aZ_m)\otimes L\bigr)=0,
$$
 hence the the constant 1 section of
$L|_{C_{\tau(1)}}\cong \sO_{C_{\tau(1)}}$ lifts to any order. By the Theorem on Formal Functions, we get that $f_*L$ is a line bundle and $L\cong f^*f_*L$. \qed

\medskip

{\it Claim \ref{lipman.prop}.4.}  $\langle \sO_T(C_i):i\in I\rangle\subset \pic(T)$ is a  subgroup of finite index
 $d\coloneqq\det (C_i\cdot C_j)$.  Denote the quotient by $H(T)$.
\medskip

Proof. In abstract terms, we have a subgroup of $\oplus_{i\in I}\mathbb{Z}[C_i]$ generated by $C_j^*\coloneqq\bigl((C_i\cdot C_j):i\in I\bigr)$. This shows that
$H(T)$ is determined by the matrix $(C_i\cdot C_j)$, which is negative definite by the Hodge index theorem.
The rest is  linear algebra.  \qed

\medskip
Putting Claims~\ref{lipman.prop}.3--4 together gives the following.
\medskip

{\it Claim \ref{lipman.prop}.5.} There is an injection
$\picl(s,S)\into H(T)$.
In particular, $|\picl(s,S)|$ divides $\det (C_i\cdot C_j)$. \qed
\end{say}
\medskip

Let us axiomatize the properties  (\ref{lipman.prop}.1--2).

\begin{definition}  \label{spr.lip.defn} Let $X$ be a normal scheme
with singular locus $V\subset X$ and such that  $V$ is regular.
A {\it Lipman-type resolution} is a resolution
$f\:Y\to X$ with exceptional divisors $\{E_i:i\in I\}$ 
and  the following properties.
\begin{enumerate}[label=(\arabic*)]
\item There are finite, flat  morphisms  $V_i\to V$ such that  $E_i$ is either  $\bP^1_{V_i}$ or a regular conic bundle with irreducible fibers in   $\bP^2_{V_i}$. 
Furthermore, every $V_i\to V$ is the composite of a purely inseparable morphism $V_i\to V^*_i$ with an \'etale morphism $V^*_i\to V$.
\item There is  a sequence of effective, $f$-exceptional cycles
$Z_1\leq \cdots\leq Z_m$ and a map  $\tau: \{1,\dots, m\}\to I$ such that
\begin{enumerate}[label=(\alph*)]
\item  $\sO_{Y}(-Z_m)$ is $f$-ample, 
\item   $Z_1=E_{\tau(1)}$ and 
\item $E_{\tau(i+1)}\cong \bP^1_{V_{\tau(i+1)}}$  for $i\geq 1$ and  we have exact sequences
$$
0\to \sO_{E_{\tau(i+1)}}(-1)\to \sO_{Z_{i+1}}\to \sO_{Z_{i}}\to 0. 
$$
\end{enumerate}
\end{enumerate}
Note that  a Lipman-type resolution is preserved by \'etale morphisms
$U\to X$ and by completions  $\hat X_x\to X$.
The only  complication is that the $E_i$ may become reducible
if $V_i\to V$ is not purely inseparable. We claim that the latter can be achieved after an \'etale base change
$X'\to X$. Indeed,  $V^*_i\to V$ is standard \'etale in a neighborhood of the generic point of $V$  by \cite[\href{https://stacks.math.columbia.edu/tag/02GT}{Tag 02GT}]{stacks-project}, hence it extends to a finite  \'etale cover
$X'_i\to X^\circ_i\subset X$. Taking the fiber product of all of these now gives
$X'\to X^\circ\coloneqq\cap_i X^\circ_i$. 
For the resulting
$f'\:Y'\to X'$ with exceptional divisors  $E'_i$, all the   $V'_i\to V'$
are purely inseparable. 

We call such an $f':Y'\to X'$ a 
{\it stabilized Lipman-type resolution.}  
For such a resolution, the $E'_i$ stay irreducible after completion.
\end{definition}

Proposition~\ref{lipman.extend.prop} is now a direct consequence of the following 2 claims.

\begin{lemma}\label{spr.lip.lem.1}
 Let $X$ be a normal  scheme.  Assume that its codimension 2 singular points are  rational.  Then there is a closed subset $W\subset X$ of codimension $\geq 3$ such that  $X\setminus W$ has a 
Lipman-type resolution.
\end{lemma} 

\begin{lemma}\label{spr.lip.lem.2}
 Let $X$ be a normal scheme that has a
Lipman-type resolution. Then 
\begin{enumerate}[label=(\arabic*)]
\item  $\Cl(X)/\pic(X)$ is finite and
\item   $X$ is analytically $\mathbb{Q}$-factorial.  Moreover, the orders of the local class groups   $\Cl(\hat X_x)$ are uniformly bounded for $x\in X$.
\end{enumerate}
\end{lemma} 

 \begin{say}[Proof of Lemma~\ref{spr.lip.lem.1}]  
Let $\{x_i:i\in I\}$ be the  codimension 2 singular points of $X$.  
The minimal resolution at $x_i$ is obtained by blowing-up an $m_{x_i,X}$-primary ideal sheaf
$I_{x_i,X}\subset \sO_{x_i,X}$. We extend these to ideal sheaves 
$I_i\subset \sO_X$ whose co-support is the closure $V_i$ of $x_i$. By blowing-up we get
$Y_i\to X$.  There are open neighborhoods  $x_i\subset U_i\subset X$ such that
$Y_i\times_XU_i\to U_i$ is a Lipman-type resolution.
Set 
$$
W\coloneqq\cup_i (V_i\setminus U_i)\bigcup \cup_{i\neq j} (V_i\cap V_j).
$$
Let $J$ be the restriction of $\cap_i I_i$ to $X\setminus W$. 
Note that the $V_i\cap (X\setminus W)$ are disjoint from each other, hence blowing up $J$ is the same as blowing up the  $I_i|_{ X\setminus W}$ one at a time. 
Thus over each $V_i\cap (X\setminus W)$ we obtain the Lipman-type resolution $Y_i\times_XU_i\to U_i$. \qed
\end{say}

 \begin{say}[Proof of Lemma~\ref{spr.lip.lem.2}]
The claims are Zariski local, we may thus assume that 
$V=\Sing X$ is irreducible.  

The claims also descend from finite \'etale base changes.
Thus, as we noted in Definition~\ref{spr.lip.defn}, we may assume that the  set of irreducible components of the exceptional divisor, and hence the group  $H(T)$ is invariant under  completions.

Note that  $L|_{E_i}$ is $f$-trivial iff its degree on the generic fiber is 0.
If $L|_{E_i}$ is $f$-trivial for every $i$ then, 
 as in Claim~\ref{lipman.prop}.4,
we can use the sequences in Definition~\ref{spr.lip.defn}.2.c to lift this trivialization to get that
$f_*L$ is a line bundle and $L\cong f^*f_*L$. 

As in Claim~\ref{lipman.prop}.5 we get injections
$\Cl(X)/\pic(X)\into H(T)$ and $\Cl(\hat X_x)\into H(T)$ for every completion. \qed
\end{say}

\subsection*{Proof of Proposition~\ref{dlt.loc.pic.prop}}{\ }

First we  take a suitable log resolution $g\:X\to (Y, \Delta)$. If $(Y, \Delta)$ is klt then we can take any resolution obtained by a sequence of blow-ups whose centers lie over $\Sing Y\cup \Sing(\Supp\Delta)$, and such that $\Ex(g)\cup g^{-1}_*\Delta$ is a  simple normal crossing divisor. 
(The most general resolution of 3-dimensional schemes  is proved in  \cite{CossartPiltantResolutionSingularitiesArithmetical3folds}; see also \cite{k-wit}.)
In this case $\Ex(g)$ supports a $g$-ample divisor $H$ and the  coefficients of $H$ can be chosen to be linearly independent over any given finitely generated subfield of $\r$. 

The general dlt case can be reduced to the klt case using \cite[2.43]{KollarMori}.

Then we run a suitable MMP to get  from $X$ to $Y$. On $X$ the completions of the local rings are regular, hence factorial. Thus it is enough to show that
the finite generation of the prime-to-$p$ part is preserved by each MMP step.

First we use that, by \cite{KollarrelativeMMPwoQfactoriality},  the MMP steps are simpler than expected, see Proposition~\ref{mmp.simpler.prop}. Then we discuss how the 
divisor class groups change under MMP steps.
 
For divisorial contractions we use 
Lemma~\ref{plt.bu.use.cor}. Note that assumption (\ref{plt.bu.use.cor}.1) is statisfied by Proposition~\ref{lipman.extend.prop} and 
assumption (\ref{plt.bu.use.cor}.2) holds by \autoref{cla.EulerCharacteristic}.

For  flips  we use Lemma~\ref{when.flip.qfact.pres}. 
In applying it,  a key  issue is the following. 
Let $g\:X\to Z$ be a 3-dimensional flipping contraction with exceptional curve $C=\Ex(g)$. In characteristic 0 we know that $R^1g_*\sO_X=0$. Thus
 $H^1(X,\sO_C)=0$, hence  $\pico(C)=0$.  In positive characteristic 
the vanishing of $R^1g_*\sO_X$ is not known and may not be true.

We go around this problem as follows.
Let   $C^{\mathrm wn}\to C$ denote the weak normalization. 
We prove in Lemma~\ref{when.flip.almost.R1} that $h^1\bigl(C^{\mathrm wn}, \sO_{C^{\mathrm wn}}\bigr)=0$. Therefore  $C$ is homeomorphic to a tree of smooth rational curves, which implies that 
$\pico(C)$ is a unipotent group scheme (hence $p^\infty$-torsion).
This is enough for our purposes. \qed

\begin{proposition} \cite{KollarrelativeMMPwoQfactoriality, Witaszekflipslowchar}\label{mmp.simpler.prop}
 Let $(Y, \Delta)$ be a 3-dimensional, klt scheme  (see Remark~\ref{future.cite}) and $g:X\to Y$ a  log resolution with exceptional divisors $E=\sum E_i$.  Assume that $E$ supports a $g$-ample divisor $H$; we can then choose its  coefficients to be  sufficiently general. 

Then   the MMP starting with  $(X^0, E^0+\Delta^0)\coloneqq(X, E+g^{-1}_*\Delta)$ with scaling of $H$ runs and  terminates with $(Y, \Delta)$.  Each  step  $X^i\dasharrow X^{i+1}$ of this MMP is
\begin{enumerate}[label=(\arabic*)]
\item either a contraction $\phi_i:X^i\to X^{i+1}$, whose exceptional set is an irreducible component of $E^i$,
\item or a flip  $X^i\stackrel{\phi_i}{\longrightarrow}Z^i  \stackrel{\psi_i}{\longleftarrow}(X^i)^+=X^{i+1}$. Moreover, 
there are irreducible components $E^i_1, E^i_2$  such that both $E^i_1,$ and $ -E^i_2$ are $\phi_i$-ample. \qed
\end{enumerate}
\end{proposition}

\medskip
Divisorial contractions are handled by the following.

\begin{lemma} \label{plt.bu.use.cor}
Let $(y,Y)$ be a normal,  local  scheme and
$\pi\:X\to Y$ a proper modification with reduced exceptional set $E$, where $X$ is normal.   
Assume that  $p\coloneqq\Char k(y)>0$ and the following hold. 
\begin{enumerate}[label=(\arabic*)]
\item There is a  finite point set $P\subset E$ such that
\begin{enumerate}[label=(\alph*)]
\item $\pic(X\setminus P)\subset\Cl(X\setminus P)$ has finite index, and 
\item  the prime-to-$p$ part of  $\picl(x, X)$ is
 finitely generated for all $x\in P$.
\end{enumerate}
\item The prime-to-$p$ part of  $\pic(E)$ is
 finitely generated.
\end{enumerate}
Then the prime-to-$p$ part of  $\Cl(Y)$ is
 finitely generated.
\end{lemma}

Proof.  Note first that $\Cl(Y)\cong \Cl(X)/\langle E_i:i\in I\rangle$, 
 where  $\{E_i\subset E: i\in I\}$ are the irreducible components that have codimension 1 in $X$. Thus our claim is equivalent to
 proving that the prime-to-$p$ part of  $\Cl(X)$ is
 finitely generated.
Note that $\Cl(X\setminus P)= \Cl(X)$ and,
by assumption (1.a), 
$\pic(X\setminus P)\subset \Cl(X\setminus P)$ has finite index.
Extending divisors from $X\setminus P $ to $X$ gives an exact sequence
$$
\pic(X)\to  \pic(X\setminus P)\to \bigoplus_{x\in P} \picl(x, X).
$$
 It remains  to show that
 the prime-to-$p$ part of  $\pic(X)$ is
 finitely generated. 
The prime-to-$p$ part of $\pic(E)$ is finitely generated by assumption (2), and the kernel of
$\pic(X)\to \pic(E)$ is  $p^\infty$-torsion by \cite{CasciniTanakaRelativeSemiAmplenessPosChar}.\qed
\medskip

(Note that  \cite{CasciniTanakaRelativeSemiAmplenessPosChar} deals with $\bF_p$-schemes; this is all we need for
Proposition~\ref{dlt.loc.pic.prop}. We plan to return to  the general case later; see also \cite{WitaszekKeel'sBPFMixedChar}. 
Also, \cite{CasciniTanakaRelativeSemiAmplenessPosChar} states only that the kernel of
$\pic(X)\to \pic(E)$ is torsion. If $L$ is in the kernel and has order $m$ not divisible by $p$, then it gives a degree $m$ \'etale cover of $X$ that is trivial on $E$, hence on $E_y$. Thus $m=1$, cf.\ \cite[3.1]{ArtinAlgebraicApproximationOfStructuiresOverCmpleteLocalRIngs} or  \cite[\href{https://stacks.math.columbia.edu/tag/0BQC}{Tag 0BQC}]{stacks-project}.)

\medskip

For flips we need a more precise version.

\begin{lemma} \label{plt.small.use.cor}
Let $(y,Y)$ be a 3-dimensional, normal, complete, local  scheme and
$\pi\:X\to Y$ a proper modification,  where $X$ is normal and its codimension 2 singular points are rational.  
Assume that  $C\coloneqq\red\pi^{-1}(y)$ is 1-dimensional and $p\coloneqq\Char k(y)>0$. 
The following are equivalent.
\begin{enumerate}[label=(\arabic*)]
\item  The prime-to-$p$ part of  $\Cl(\hat X_x)$ is
 finitely generated for all $x\in C$, and  $h^1\bigl(C^{\mathrm wn}, \sO_{C^{\mathrm wn}}\bigr)=0$.
\item The prime-to-$p$ part of  $\Cl(Y)$ is
 finitely generated.
\end{enumerate}
\end{lemma}

Proof. The proof of  (1) $\Rightarrow$ (2)  goes as in
(\ref{plt.bu.use.cor}). 

For the converse the interesting part is to show that the prime-to-$p$ part of  $\Cl(\hat X_x)$ is
 finitely generated for all $x\in C$.
As we noted in Remark~\ref{loc.pic.cl.comp.rem}, it is enough to prove 
that the prime-to-$p$ part of  $\picl(x,\hat X_x)$ is
 finitely generated.

 Let
$D_x\subset \hat X_x$ be a divisor that is Cartier outside $x$. Then it is linearly equivalent to a divisor $D'_x \subset \hat X_x$ such that
$C\cap D'_x=\{x\}$. Since $X$ is complete in the topology induced by the ideal sheaf of $C$, we can view $D'_x$ as a divisor on $X$. That is, there is a
surjection
$$
\pic(X\setminus \{x\})\onto \picl(x, \hat X_x).
$$
Thus if the prime-to-$p$ part of  $\picl(y, Y)$ is
 finitely generated, then so is the 
prime-to-$p$ part of  $\picl(x, \hat X_x)$. \qed

\medskip
Applying (\ref{plt.small.use.cor}) twice gives the following.

\begin{corollary} \label{when.flip.qfact.pres} Consider a 3-dimensional flip
$$
(C\subset X,\Delta)\stackrel{g}{\longrightarrow} Z\stackrel{\hphantom{aa}g^+}{\longleftarrow} (C^+ \subset X^+ , \Delta^+).
$$  Assume that 
the prime-to-$p$ part of  $\Cl(\hat X_x)$ is
 finitely generated for all $x\in C$ and that  $h^1\bigl(C^{\mathrm wn}, \sO_{C^{\mathrm wn}}\bigr)=0$. Then the prime-to-$p$ part of  $\Cl(\hat X^+_{x^+})$ is
 finitely generated for all $x^+\in C^+$. \qed
\end{corollary}

It remains to show that the condition  $h^1\bigl(C^{\mathrm wn}, \sO_{C^{\mathrm wn}}\bigr)=0$ holds  for the flips in 
Proposition~\ref{mmp.simpler.prop}.2.

\begin{lemma} \label{when.flip.almost.R1}
Let $g:( X,\Delta)\to Z$ be a 3-dimensional flipping contraction with exceptional curve $C=\Ex(g)$. Assume that there is an irreducible component $S\subset \rdown{\Delta}$ such that 
$C\subset S$.  
Then  $h^1\bigl(C^{\mathrm wn}, \sO_{C^{\mathrm wn}}\bigr)=0$.
\end{lemma}

Proof.   The normalization $\bar S\to S$ is a  homeomorphisms by (\ref{loc.conn.1.lem}). Let $\bar C\subset \bar S$ denote the preimage of
$C$. Then $\bar C$ is the exceptional set of an 
$\bigl(\bar S, \diff_{\bar S}(\Delta-S)\bigr)$-negative  contraction, hence $h^1(\bar C, \sO_{\bar C})=0$.  Furthermore,
$\bar C$ is  weakly normal and $\bar C\to C$  is a  homeomorphisms.
Thus $\bar C=C^{\mathrm wn}$. \qed

\begin{lemma}\label{loc.conn.1.lem}
 Let $(X,D+\Delta)$ be a dlt pair with $D$ irreducible and $\mathbb{Q}$-Cartier.
Then the normalization $\bar D\to D$ is a universal homeomorphism.
\end{lemma}

Proof. By \cite[2.31]{KollarSingulaitieofMMP} $D$ is regular outside a subset of codimension $\geq 2$.   The rest follows from  \cite[XIII.2.1]{SGA2}. \qed

\subsection*{Acknowledgments} I thank J.~Carvajal-Rojas, A.~St{\"a}bler, J.~Witaszek and C.~Xu
for many comments and helpful discussions.
Partial  financial support    was provided  by  the NSF under grant number
DMS-1901855.

\fancyhf{}
\fancyhead[LE,RO]{\thepage}
\fancyhead[CE]{}
\fancyhead[CO]{}
\renewcommand\headrulewidth{0pt}
\pagestyle{fancy}

\bibliographystyle{skalpha}
\bibliography{MainBib}

\end{document}